\documentclass[12pt]{article}
\usepackage{graphicx}
\usepackage{natbib} %comment out if you do not have the package
\usepackage{url} % not crucial - just used below for the URL 

\usepackage{amsthm}
\usepackage{amsmath}
\usepackage{mathtools}
\usepackage{amssymb}
\usepackage{dsfont}
\usepackage{stmaryrd}

\usepackage{algorithm}

\usepackage{hyperref}
\usepackage{cleveref}
\usepackage{subcaption}
\usepackage{IEEEtrantools}
\usepackage{xcolor}

\DeclareMathOperator*{\argmin}{arg\,min}
\DeclareMathOperator*{\argmax}{arg\,max}
\newtheorem{theorem}{Theorem}

\newcommand{\blind}{0}

\begin{document}

\def\spacingset#1{\renewcommand{\baselinestretch}%
{#1}\small\normalsize} \spacingset{1}

%%%%%%%%%%%%%%%%%%%%%%%%%%%%%%%%%%%%%%%%%%%%%%%%%%%%%%%%%%%%%%%%%%%%%%%%%%%%%%

\if0\blind
{
  \title{\bf An Information Geometry Approach to Robustness Analysis for the Uncertainty Quantification of Computer Codes}
  \author{Clement GAUCHY$^{1}$\footnote{Present affiliation: DES/ISAS - Service d'études mécaniques et thermiques (SEMT), CEA Saclay, 91191 Gif-sur-Yvette, France} \\
  %\thanks{    The authors gratefully acknowledge \textit{please remember to list all relevant funding sources in the unblinded version}}\hspace{.2cm}\\
    Jerome STENGER$^{12}$\\
    Roman SUEUR$^{1}$\\
    Bertrand IOOSS$^{12}$\footnote{Corresponding author: bertrand.iooss@edf.fr}\\
    \\
    $^1$EDF R\&D, D\'epartement PRISME, 6 Quai Watier, 78401, Chatou, France\\
    $^2$Institut de Math\'ematiques de Toulouse, 31062 Toulouse, France}
  \maketitle
} \fi

\if1\blind
{
  \bigskip
  \bigskip
  \bigskip
  \begin{center}
    {\LARGE\bf An Information Geometry Approach to Robustness Analysis for the Uncertainty Quantification of Computer Codes}
\end{center}
  \medskip
} \fi

\bigskip
\begin{abstract}

%\modif{Mettre en rouge les modifs non mineures que l'on fait, avec la commande \text{\textbackslash modif\{\}}.}

%\comment{Mettre en bleu les commentaires que l'on fait, avec la commande \text{\textbackslash comment\{\}}.}

   Robustness analysis is an emerging field in the uncertainty quantification domain. It involves analyzing the response of a computer model---which has inputs whose exact values are unknown---to the perturbation of one or several of its input distributions. Practical robustness analysis methods therefore require a coherent methodology for perturbing distributions; we present here one such rigorous method,  based on the Fisher distance on manifolds of probability distributions. Further, we provide a numerical method to calculate perturbed densities in practice which comes from Lagrangian mechanics and involves solving a system of ordinary differential equations. The method introduced for perturbations is then used to compute quantile-related robustness indices. We illustrate these  ``perturbed-law based'' indices on several numerical models.
   We also apply our methods to an industrial setting: the simulation of a loss of coolant accident in a nuclear reactor, where several dozen of the model's physical parameters are not known exactly, and where  limited knowledge on their distributions is available.

\end{abstract}

\noindent%
\vfill

\newpage

\spacingset{2} % DON'T change the spacing!

%%%%%%%%%%%%%%%%%%%%%%%%%%%%%%%%%%%%%%%%%%%%%%%%%%%%%%%%%%%%%%%%%%%%%%%%%
\section{Introduction}
\label{sec:intro}

Over the last few decades, two major trends in industrial and research practices have led to a rise in the importance of uncertainty quantification (UQ) methods \citep{derdev08,smi14,ghahig17}. The first is the replacement of full-scale physical experiments, considered costly and difficult to implement, by numerical models. This choice raises the issue of a potential mismatch between computer codes and the physical reality they aim to simulate. The second trend involves accounting for risk in an increasing number of industrial activities, and in particular, doing it quantitatively.

In both situations, UQ can be performed by considering the imperfectly known inputs of a computer code (given as a function $G(\cdot)$) as a vector of 
random variables $\textbf{X} = (X_1,\ldots,X_d)$. 
The most widespread approach consists of running $G(\cdot)$ with different combinations of inputs in accordance with their range of plausible values, in order to study the associated uncertainty found in the output $Y=G(X_1,\ldots,X_d)$, or, in other settings, to estimate a specific quantity of interest (QoI). The latter means some statistical quantity derived from $Y$, e.g,. the performance, via the mean of $Y$, or some risk criterion in terms of a high-level quantile.

For instance, the nuclear industry faces major issues as facilities age and regulatory authorities' requirements toughen \citep{bucpet10,mouwil17}. 
One precise case is where operators have to study \emph{loss of coolant accidents} (LOCA), which result in a break in the primary loop of pressurized water nuclear reactors.
This scenario can be simulated using system thermal-hydraulic computer codes which involve dozens of physical parameters including such things as condensation and heat transfer coefficients \citep{mazvac16,sansan18}. 
However, the values of many such parameters are known with only limited precision \citep{lar19} as they are typically calculated by way of other quantities measured via small-scale physical experiments.
%Other uncertainties might come from the configuration of the boiler at the beginning of the accidental transient. 
%Indeed, the state of the reactor can vary throughout time in accordance with the constraints of the production process (for example the power level) or external parameters (atmospheric conditions). 
%or design parameters (setpoint values of automatically regulated quantities{\color{red} je comprends pas ce terme}). 
Certain other variables can only be measured during periodic inspections, e.g., the characteristics of pumps in hydraulic systems.

Various methods from the UQ domain can be useful when dealing with such uncertainties in system safety analyses. 
First, certain methods aim to improve the exploration of the input domain $\mathcal{X}$ by using design of experiments, e.g., space filling designs \citep{fanli06}.
These make it possible to cover an input domain as evenly as possible for a fixed number of code runs, and to limit unexplored areas.
%These are called , of which we can find maximin LHS and low discrepancy designs. 
%These last ones are also used to accelerate the estimation of quantities derived from the output distribution. 
For the estimation of certain specific QoI such as the probability of threshold exceedance or an $\alpha$-order quantile of the output, Monte Carlo-type methods are often preferred.
In particular, accelerated Monte Carlo methods (e.g., importance sampling or subset simulation) target the most informative areas of $\mathcal{X}$ in the sampling algorithm in order to estimate the QoI while controlling its estimation error \citep{morbal16}.
As a preliminary or concomitant stage, global sensitivity analysis is also essential in order to eliminate non-influential parameters and rank influential ones according to their impact on the QoI \citep{ioolem15,ioomar19}. 

%However, in most studies, computing efficiently the QoI is not sufficient to make a decision albeit difficult to get as requiring a large number of computer experiments. 
%A last use of UQ consists in estimating the law of uncertain inputs while taking into account the impact on the output of a limited knowledge on this input. This methods are called uncertainty calibration as presented in [Kennedy M., and O’Hagan A., Bayesian calibration of computer models, J.R. Statis. Soc, 2001)].

All of these approaches are useful for dealing with the existence of uncertainty in applied problems. 
However, industrial (e.g., nuclear facility) operators face the difficulty of justifying their risk assessment methods over and above providing simulation results. 
Risk assessment methods must demonstrate that they provide reliable safety margins. Doing so means a probable overestimation of the likely risk after accounting for all known sources of uncertainty.
This principle of conservatism, which can be easily implemented when dealing with very simple monotonic physical models, can be hard in practice to adapt to computer codes simulating complex and non-monotonic physical phenomena. 
Indeed, it is rarely simple to apply this principle when implementing UQ methods based on a set of computer experiments that provide a whole range of values for the output quantity $Y$.

To address this issue, a new UQ branch of robustness analysis has emerged in recent years in the field of sensitivity analysis. % \citep{ioo18}. 
It consists of evaluating the impact of the choice of the inputs' distributions and, more precisely, involves an analysis of how values of the QoI vary with respect to these choices. 

An initial strategy is to consider a whole set of input distributions and analyze the related output distributions.
For global sensitivity analysis, \citet{hart2018robustness} use ``optimal perturbations'' of probability density functions to analyze the robustness of  variance-based sensitivity indices (called Sobol indices \citep{sob93}).
\citet{meynaoui} and \citet{chabal18} propose approaches that deal with ``second-level'' uncertainty, i.e., uncertainty in the parameters of the input distributions.
An alternative approach known as ``optimal uncertainty quantification''  avoids specifying the input probability distributions, transforming the problem into the definition of constraints on moments \citep{owhadi_optimal_2013,stegam19}.
The latter is out of the scope of the present work; here we consider that the initial probability distributions input by the user remain important.

In real-world uncertainty quantification studies in engineering, input distributions are typically truncated, as we are dealing with  physical parameters that in reality belong to finite-dimensional compact sets. It is therefore natural to assume that there is no uncertainty in the support of input random variables. In this paper, we also assume that the inputs are mutually independent.
Keeping in mind that our goal is to directly deal with input distributions (and not to consider second-level uncertainty), one particularly interesting solution has been proposed in the context of reliability-oriented sensitivity analysis by \citet{lem14} (see also \citet{lemser15,suebou16}): perturbed-law based indices (PLI).
A density perturbation consists of replacing the density $f_i$ of one input $X_i$ by a perturbed one $f_{i\delta}$, where $\delta \in \mathbb{R}$ represents a shift in the value of a specific moment (e.g., the mean or variance). 
Among all densities which have an equivalent shift in the chosen moment, $f_{i\delta}$ is defined as that which minimizes the Kullback-Leibler divergence from $f_i$.
This method has been applied to compute a QoI that corresponds to a failure probability \citep{iooleg19,perdef19},  a quantile \citep{sueioo17,lar19} and a superquantile \citep{iooss_robustness_2020,largau20}.

However, this method is not fully satisfactory. Indeed, the interest of perturbing moments is arguable: perturbing the mean or variance of input distributions is essentially an arbitrary decision. Moreover, it has been observed that applying the same perturbation to two different parametric distributions can yield significant differences in the value of the Kullback-Leibler divergence (between the initial and perturbed density). 
%\modif{the moment choice to be disturbed remains arbitrary. Thus, it seems impossible to compare a disturbed density in mean and a disturbed density in variance. Moreover, some probability densities do not even have defined mean or variance. This weakens the generalization of this methodology} 

Another possibility described in \citet{perdef19} is to use an iso-probabilistic operator to transform all input random variables into  standard Gaussian ones.
This makes perturbations comparable when applied in this new space, but it remains difficult to translate this interpretation back into the initial physical space, that which is of interest to users.
We note in passing another type of robustness analysis that has been proposed in quantitative finance by \citet{cont:hal-00413729}. These authors investigate whether the estimated QoI is sensitive to small perturbations of the empirical distribution function. To this end, they define the robustness of a QoI in terms of its continuity with respect to the Prokhorov distance on the set of integrable random variables. 

The goal of the present paper is to propose a novel approach to perturbing probability distributions.
It relies on density perturbations based on the Fisher distance \citep{Costa} as a measure of dissimilarity between an initial density $f_i$ and a perturbed one $f_{i\delta}$. 
This distance defines a geometry on spaces of probability measures known as \emph{information geometry} \citep{Nielsen2013}. 
The statistical interpretation of the Fisher distance provides an equivalence between perturbations of non-homogeneous quantities, and consequently a coherent framework for robustness analysis. 
Before presenting this approach, we first review existing density perturbation methods in \Cref{Section 2}. 
\Cref{Section 3} is dedicated to the description of our method and a discussion of numerical methods.
%To illustrate our methodology, we investigate in \Cref{Section 4} two analytical applications and an industrial one based on a thermal-hydraulic modeling of a nuclear reactor submitted to an accidental transient (same scenario than in \citet{stegam19}). The last section gives a conclusion and some research perspectives.%
\Cref{Section 4} illustrates our density perturbation method for the practical robustness index (PLI). 
We apply this to a toy example function and an industrial case study in \Cref{Section 5}. This is followed by conclusions and research options going forward.
The online supplementary materials contain four further sections describing: 1) some numerical results on Fisher spheres computations; 2) the \emph{reverse importance sampling} method; 3) the theoretical properties of the PLI-quantile estimator; and 4) an application of the PLI to another toy model.

%%%%%%%%%%%%%%%%%%%%%%%%%%%%%%%%%%%%%%%%%%%%%%%%%%%%%%%%%%%%%%%%%%%
\section{Previous approaches to density perturbation for uncertainty quantification robustness analysis}
\label{Section 2}

The method of \citet{lemser15}, later known as PLI by \citet{suebou16}, is based on the idea of perturbing input densities.
It aims to provide a practical counterpart to the general idea of analyzing the output QoI of a model in a UQ framework when one or several parameters of the input probabilistic model (considered as the reference) is/are modified. 
This strategy can be seen as a way to take into account an ``error term'' that could be added to an imperfectly-known input distribution.

%%%%%%%%%%%%%%%%%%%%%%%%%%%%
\subsection{Kullback-Leibler divergence minimization}
\label{Section 2.1}

To construct a perturbed distribution $f_{i\delta}$ from a distribution $f_i$, the approach of \citet{lemser15} is nonparametric. This method essentially aims to analyze perturbations in the most common characteristics of input distributions such as the mean or variance. To illustrate this for the mean, let us assume the random variable $X_i\sim  f_i$ has mean $\mathbb{E}[X_i]=\mu$. By definition, the perturbed density will have a mean of $\mu+\delta$. However, this is clearly insufficient to fully determine the perturbed distribution and in particular explicitly access the value of $f_{i\delta}$ over the whole domain of $X_i$. Amongst all densities with a mean of $\mu+\delta$, $f_{i\delta}$ is defined as the solution of the minimization problem:
\begin{equation}
f_{i\delta} = \argmin\limits_{\pi \in \mathcal{P},\;
s.t \: \mathbb{E}_{\pi}[\mathbf{X_i} ] = \mathbb{E}_{f_i}[X_i] + \delta} KL(\pi|| f_i) \ ,
\end{equation}
where $\mathcal{P}$ is the set of all probability measures which are absolutely continuous with respect to $f_i$. This approach basically consists of perturbing the chosen parameter while changing the initial model as little as possible. Under this definition, ``changing'' the model corresponds to increasing the entropy, where the Kullback-Leibler divergence between two densities $f$ and $\pi$ is given by: 
\begin{equation}
KL(\pi||f) = \int \log\left(\frac{\pi(x)}{f(x)}\right)f(x) \,dx \ .
\end{equation}
This method can be applied to higher-order moments (for instance moments of order 2---to define a variance perturbation) and, more generally, to constraints that can be expressed as a function of the perturbed density, such as quantiles \citep{lem14}. 
Notice that in the case of an initial Gaussian distribution, the perturbed distribution remains Gaussian with simply a change in the mean or variance. 

In the general case, this method has several drawbacks. First, the likelihood ratio between $f_{i\delta}$ and $f_i$ (see section \ref{Section 4.1}) may not have an analytic form, which leads to numerical difficulties. Second, this method requires that moments be defined for the initial density.
Third, the main difficulty concerns the interpretation of results. Indeed, each input containing uncertainty in the UQ model is perturbed using a range of values of $\delta$. However, the choice of the moment to be perturbed remains arbitrary. Such problems are particularly worrisome when uncertainty in the input parameter is epistemic \citep{KIUREGHIAN2009105} (meaning it comes from a lack of knowledge). Perturbing only the mean or variance of such an input distribution may be overly limiting when trying to satisfy regulatory authorities focused on being overly conservative. 

We also recall that all input random variables are assumed to be mutually independent. Nonetheless, the effect of perturbations can be considered only for each variable individually, and in absolute terms (the same $\delta$ shift may have a quite different impact on different input densities). This method thus proves challenging when one wants to compare the relative impact of perturbations on different inputs.

%%%%%%%%%%%%%%%%%%%%%%%%%%%%
\subsection{Standard space transformations}
\label{Section 2.2}

To interpret a $\delta$ shift on an input distribution and in particular to enable inputs to be compared with respect the impact on the QoI of the same perturbation, an equivalence criterion between inputs is required. An idea developed by \citet{perdef19} consists of applying perturbations in the so-called \emph{standard space} (instead of the initial physical space) in which all input distributions are identical, thus making all perturbations equivalent. 
Then the perturbed densities are obtained by applying a reverse transform. 

In the case of independent inputs, the required distribution transform is a simple inverse probability one. Given a real-valued random vector $\textbf{X} = (X_1,\ldots,X_d)^T$ with cumulative distribution function (cdf) $F$, the transform is the random vector $\textbf{S} = \Phi^{-1}(F(\textbf{X} ))$, where $\Phi$ is the cdf of the standard Gaussian distribution $\mathcal{N}(\textbf{0} , \textbf{I}_d )$. Consequently, $\textbf{S}$ follows a standard Gaussian distribution whatever the initial distribution $F$. In the Kullback-Leibler divergence minimization framework (see \Cref{Section 2.1}), a perturbation of the mean in the standard space simply consists of a mean shift without the standard deviation changing. For instance, for a mean shift of $\delta \in \mathbb{R}^d$, the perturbed distribution is defined by the inverse transformation $F_{\delta} = F^{-1}(\Phi(\textbf{S} + \delta))$ in order to get back from the standard space to the physical space. This therefore leads to an analytic expression for the perturbed density $f_{i\delta}$ in the one-dimensional case via the change of variable formula \citep[p.318]{Stirzaker2003}:

\begin{equation}
f_{i\delta}(x) = e^{\frac{- \delta^2 +2\delta\Phi^{-1}(F_i(x))}{2}}f_i(x) \; ,
\label{eq:rosenblatt}
\end{equation}
where $F_i$ is the cdf of the variable $X_i$ for $i \in \{1,\ldots,d\}$. This simple formula makes the perturbed density and likelihood ratio easy to compute. 

However, similar perturbations in the standard space can lead to quite different ones in the physical space, depending on the initial distributions.
As an example, Figure \ref{fig: kullback} depicts the Kullback-Leibler divergence (approximated with Simpson's rule \citep{abraste74}) of two specific distributions (the Triangular\footnote{The triangular distribution $\mathcal{T}(-1,0,1)$ is parametrized by its minimum $a$, mode $b$, and maximum $c$.} $  \mathcal{T}(-1,0,1)$ and the Uniform  $\mathcal{U}[-1, 1]$) with their respective distributions in the standard space as the value of $\delta$ changes.
We see that the Kullback-Leibler divergence can behave very differently for different distributions in the physical space, even though the exact same perturbation has been applied in the standard space.
Furthermore, there is no general rule for estimating the mean of a perturbed physical input in its physical space for a given mean perturbation of the input in the standard space. Such difficulties are amplified when considering perturbations involving moments other than the mean.
For instance, there is no general equivalence in the physical space between perturbations applied to the mean and standard deviation of the same input probability distribution in the standard space.
Overall, it would seem generally difficult, if not impossible, to have a simple way to convert results provided by this method into a relationship between input and output physical quantities, making such results difficult to interpret. 

\begin{figure}[!ht]
    \centering
    \includegraphics[scale = 1]{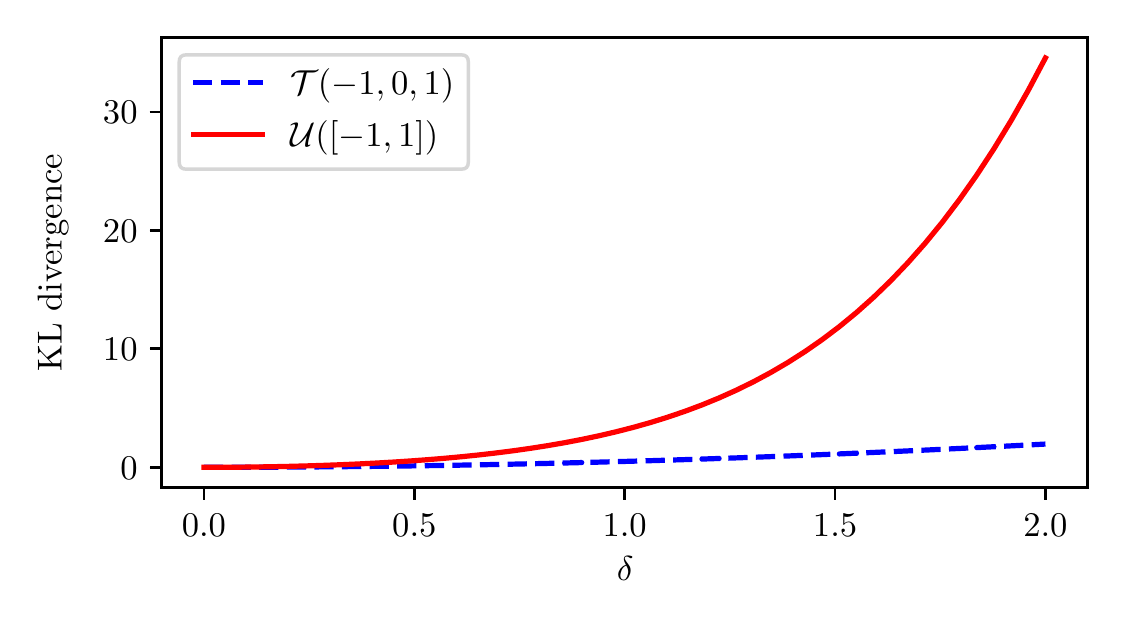}
    \caption{The Kullback-Leibler divergence between the initial distribution and the perturbed one for perturbations of the mean, $\delta \in [0, 2]$. Further details on perturbed distributions can be found in \cite{lem14} and \cite{lemser15}.}
    \label{fig: kullback}
\end{figure}

%%%%%%%%%%%%%%%%%%%%%%%%%%%%%%%%%%%%%%%%%%%%%%%%%%%%%%%%%%%%%%%%%%%%%
\section{A perturbation method based on information geometry}
\label{Section 3}

The Kullback-Leibler divergence can be interpreted as the power of a hypothesis test with null hypothesis: ``$X_i$ follows the distribution $f_i$'' and alternative hypothesis: ``$X_i$ follows the distribution $f_{i\delta}$'' \citep{Eguchi}.  For this reason, it would appear to be an appropriate tool to measure how far a perturbed density is from the initial one and thus provides a formal counterpart to the vague idea of ``uncertainty in the distribution''.
However, in a practical robustness analysis context, its use often implies embedding inputs in a standard Gaussian space to begin with, in order to allow a comparison of perturbations of different input distributions, as well as to enable easy computation of perturbed densities. This transformation of the inputs leads to manipulating perturbations in a non-physical input space, making it impossible to interpret them in terms of how they modify physical inputs' densities.

%%%%%%%%%%%%%%%%%%%%%%%%%%%%
\subsection{The Fisher distance}
\label{Section 3.1}

To try to form an intuitive understanding of the consequences of such  perturbations on the distribution of the output variable $Y$, it is necessary to base our perturbation method on a metric which allows us to compare perturbations coming from different inputs distributions in the UQ model. For instance, if one input is Gaussian and another is log uniform, their respective perturbations (associated with the same level of perturbation) should be able to be interpreted in the same way.
In particular, the perturbation method should result in identical perturbed densities when applied to two different parametric representations of the same input distribution.
The Fisher distance \citep{rao} is consistent with these wishes. It is based on the local scalar product induced by the Fisher information matrix in a given parametric space, and defines a Riemannian geometry on the corresponding set of probability measures as on any Riemannian manifold with its associated metric.

Consider the family of parametric densities $\mathcal{S} = \{f_{\theta}, \theta \in \Theta \subset \mathbb{R}^r\}$. We recall that every input variable represents a physical parameter with known domain of validity, and therefore, for all $\theta$ in $\Theta$, the support of $f_{\theta}$ is assumed to be a compact set of $\mathbb{R}$. The metric associated with the coordinate function $\theta$, called the Fisher (or Fisher-Rao) metric, is defined as:
$$
I(\theta_0) = \mathbb{E}\left[\nabla_{\theta}\log f_{\theta}(X)|_{\theta=\theta_0}(\nabla_{\theta}\log f_{\theta}(X)|_{\theta=\theta_0})^T\right] \ ,
$$
where $I(\theta_0)$ is the Fisher information matrix evaluated at $\theta_0$ for this statistical model.
The Fisher information, well-known in the fields of optimal design, Bayesian statistics, and machine learning, is a way of measuring the amount of information that an observable random variable $X$ carries about an unknown parameter $\theta$ of the distribution of $X$. 
The Fisher information matrix defines the following local inner product in $\mathcal{S}$ for $u \in \mathbb{R}^r$ and $v \in \mathbb{R}^r$:
\begin{equation}
   \langle u, v \rangle_{\theta} = u^TI(\theta)v \ .
  \label{inner}
\end{equation}

Given two distributions $f_{\theta_0}$ and $f_{\theta_1}$ in a manifold $\mathcal{S}$, a path from $f_{\theta_0}$ to $f_{\theta_1}$ is a piecewise smooth map $q: [0, 1] \rightarrow \Theta$ satisfying $q(0) = \theta_0$ and $q(1) = \theta_1$. Its length \citep{sternberg1999} $l(q)$ satisfies the following equation:
\begin{equation}
l(q) = \int\limits_{0}^{1}  \sqrt{\langle \dot q(t), \dot q(t) \rangle_{q(t)}}dt \ ,
\end{equation}
where $\dot q$ is the derivative of $q$.
Similarly, the energy \citep{sternberg1999} $E(q)$ of a path is defined by the equation:
\begin{equation}
E(q) = \int_{0}^{1} \frac{1}{2}  \langle \dot q(t), \dot q(t) \rangle_{q(t)}dt \ .
\end{equation}
The Fisher distance between $f_{\theta_0}$ and $f_{\theta_1}$is defined as the minimal length over the set of paths $\mathcal{P}(f_{\theta_0}, f_{\theta_1})$ from $f_{\theta_0}$ to $f_{\theta_1}$:
\begin{equation}
d_F(f_{\theta_1}, f_{\theta_2}) = \inf_{q \in \mathcal{P}(f_{\theta_1}, f_{\theta_2})} l(q) \ .
\end{equation}
The path $\gamma$ minimizing this length---or equivalently minimizing the energy---is called a geodesic \citep{Costa}. The specific choice of the Fisher information matrix for a Riemannian metric matrix leads to a very interesting statistical interpretation, as shown in \citet[p.27]{amari};
it is directly related to the Cramer-Rao lower bound \citep{rao} which states that, for any unbiased estimator $\widehat{\theta}$ of $\theta$, the covariance matrix $\mbox{Var}(\widehat{\theta})$ is bounded by $I(\theta)^{-1}$. This means that the Fisher information is the maximum amount of information about the value of a parameter one can extract from a given sample. More formally, under some regularity conditions (given by \cite[Theorem 3.3]{newey_chapter_1994}), if $x_1,\ldots,x_n$ are $n$ independent observations distributed according to a density $f_{\theta}$, the maximum likelihood estimator $\widehat{\theta}_n$ of $\theta$ converges weakly to a Gaussian distribution with mean $\theta$ and covariance $\displaystyle \frac{I(\theta)^{-1}}{n}$. 
The density of $\widehat{\theta}_n$, denoted $p(\widehat{\theta}_n, \theta)$ is written: 
\begin{equation}
p(\widehat{\theta}_n, \theta) = \frac{1}{\sqrt{(\frac{(2\pi)}{n})^r \det( I(\theta)^{-1})}}\exp{\left(-\frac{n(\widehat{\theta}_n - \theta)^TI(\theta)(\widehat{\theta}_n - \theta)}{2}\right)} \ .
\end{equation}
When $n \rightarrow +\infty$, this probability density is proportional to $(\widehat{\theta}_n - \theta)^TI(\theta)(\widehat{\theta}_n - \theta)$ \citep{amari}, which is the local inner product defined in equation~\eqref{inner}. This result can be interpreted as follows: the Fisher distance between two distributions $f_{\theta}$ and $f_{\theta'}$ represents the separability of the two distributions by a finite sample of independent observations sampled from the  distribution $f_{\theta}$ \citep{amari}.

Let us illustrate the Fisher distance in a simple example. Consider the statistical manifold of univariate normal distributions $\mathcal{S} = \{\mathcal{N}(\mu, \sigma^2), (\mu, \sigma) \in \mathbb{R} \times \mathbb{R}_{+}^{*}\}$. The Fisher information matrix has the analytical form \citep{Costa}:
\begin{equation}
I(\mu, \sigma) = \begin{pmatrix}
1/\sigma^2 & 0 \\
0 & 2/\sigma^2
\end{pmatrix}.
\end{equation}
If we apply the change of coordinates $\displaystyle \phi (\mu, \sigma) \rightarrow (\frac{\mu}{\sqrt{2}}, \sigma)$, the related geometry is the hyperbolic geometry in the Poincar\'e half-plane \citep{Stillwell} in which the geodesic and distance between two normal distributions are known analytically \citep{Costa}. Geometrically, geodesics are the vertical lines and half-circle centered on the line $\sigma = 0$. 
Further details on interpretations of information geometry can be found in \cite{Costa}.
 
Figure \ref{fig: geodesics} shows the position of four Gaussian distributions in the $\displaystyle (\frac{\mu}{\sqrt{2}}, \sigma)$ half-plane.
It is clear that the distributions $C$ and $D$ are more difficult to distinguish than $A$ and $B$, though in both cases the 2-Wassertein distance $W_2$ \citep{villani} is the same. The hyperbolic geometry induced by the Fisher information provides a representation in accordance with this intuition. Indeed, the two dashed curves are the geodesics between points $A$ and $B$, and $C$ and $D$. We observe that the Fisher distance between $A$ and $B$ is greater that that between $C$ and $D$. This illustrates how information geometry provides a well-grounded framework to measure statistical dissimilarities in probability measure spaces.

\begin{figure}[!ht]
  \begin{subfigure}{0.49\textwidth}
      \centering
      \includegraphics[width=\textwidth]{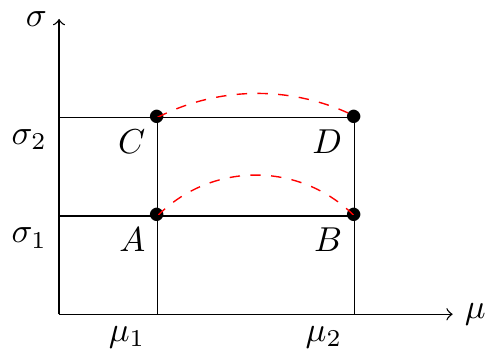}
  \end{subfigure}
  \begin{subfigure}{0.49\textwidth}
      \centering
      \includegraphics[width=0.8\textwidth]{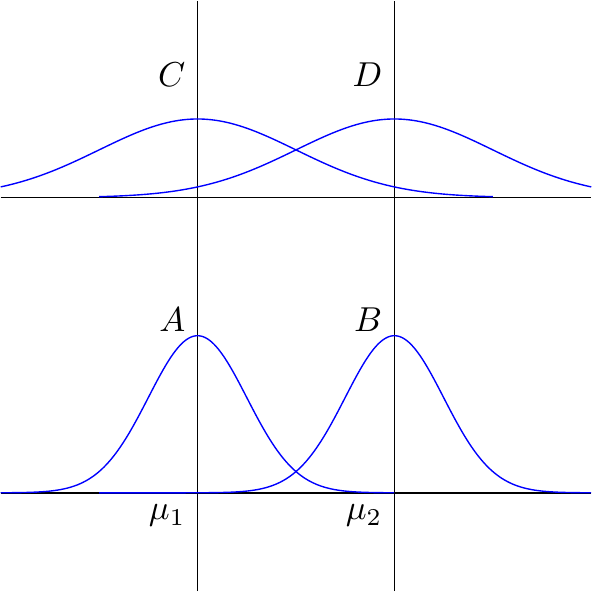}
  \end{subfigure}
  \caption{Four Gaussian distributions represented in the parameter space (left) and their respective distributions (right). Although $W_2(A, B) = W_2(C, D)$, it is easier to distinguish $A$ from $B$ than $C$ from $D$. The dashed curved lines are geodesics of different length in the $(\frac{\mu}{\sqrt{2}}, \sigma)$ plane.}
  \label{fig: geodesics}
\end{figure}

In summary, the Fisher distance provides a satisfactory and coherent grounding for our notion of density perturbation. Before continuing, let us define a perturbation of a density $f$ to be of level $\delta$ if the Fisher distance between $f$ and the perturbed density $f_{\delta}$ is equal to $\delta$. 
The set of all perturbations of $f$ at level $\delta$ is then the Fisher sphere of radius $\delta$ centered at $f$, whenever this perturbation is applied to one or another of the parameters. This implies that, in this framework, we do not consider one specific perturbed distribution, but rather, a non-finite set of probability densities. The following section is dedicated to the development of a numerical method to compute Fisher spheres of radius $\delta$ centered at $f$.

%%%%%%%%%%%%%%%%%%%%%%%%%%%%
\subsection{Computing Fisher spheres}

As detailed in \Cref{Section 3.1}, geodesics are defined as the solution of a minimization problem. Specifically, a geodesic is a path with minimal length or energy (denoted $E$). Given a smooth map $q: [0, 1] \rightarrow \mathcal{S}$, we have
\begin{equation}
E(q) = \int_{0}^{1} \frac{1}{2}  \langle \dot q(t), \dot q(t) \rangle_{q(t)}dt \ .
\end{equation}
In the following we denote $\displaystyle L(t, q, \dot q) = \frac{1}{2} \langle \dot q(t), \dot q(t) \rangle_{q(t)}$ the Lagrangian of the system. The energy of a path can be rewritten as
\begin{equation}
E(q) = \int_{0}^{1} L(t, q , \dot q) dt \ .
\end{equation}
A necessary condition for the path $q$ to minimize the energy $E$ is to satisfy the Euler-Lagrange equation (see \citet{gelfand} for details):
\begin{equation}
\frac{\partial L}{\partial q} = \frac{d}{dt}\left(\frac{\partial L}{\partial \dot q}\right) \ .
\end{equation}
We denote $\displaystyle p = \frac{\partial L}{\partial \dot q}$ and obtain by differentiation of the quadratic form $L(t, q, \dot q) = \frac{1}{2}\dot{q}^T I(q) \dot{q}$ that $p = I(q)\dot{q}$ and $\dot{q} = I^{-1}(q)p$. Then, inspired by Lagrangian mechanics theory \citep[p.65]{Arnold}, the Hamiltonian $H(p,q)$ defined by
\begin{equation}
    \begin{array}{rcl}
    H(p, q) & = & p^T\dot q - L(t, q, \dot q) = p^T I^{-1}(q) p - \frac{1}{2}\dot{q}^T I(q) \dot{q} \\
            & = & \frac{1}{2} p^T I^{-1}(q) p 
    \end{array}
\label{hamil}
\end{equation}
is constant whenever $q$ is a geodesic. Eq.~\eqref{hamil} is derived from the Euler-Lagrange equation, and implies that $(p,q)$ follows a system of ordinary differential equations (ODE) known as Hamilton's equations:
\begin{equation}
    \left \{
\begin{array}{rcccl}
    \dot q & = & \displaystyle \frac{\partial H}{\partial p} & = & I^{-1}(q)p \,,\\
    \dot p & = & - \displaystyle \frac{\partial H}{\partial q} & = & \displaystyle\frac{\partial L(t, q, I^{-1}(q)p)}{\partial q}\,.
\end{array}
\right.
\label{hamileq}
\end{equation}

The goal is to find any geodesics $q$ that satisfy $q(0) = \theta_0$ and $d_F(f, q(1)) = \delta$. This corresponds to computing the Fisher sphere centered at $f_{\theta_0}$ with radius $\delta$. The only degree of freedom remaining to solve the ODE system~\eqref{hamileq} entirely is the initial velocity $p(0)$. Notice that the Hamiltonian is equal to the kinetic energy, since $p = I(q)\dot q$. As the Hamiltonian is constant on a geodesic, we have for all $t$:
\begin{equation}
\frac{1}{2} \langle \dot q(t), \dot q(t) \rangle_{q(t)} = k \ ,
\end{equation}
where $k$ is non-negative. The length of $q$ is therefore equal to
\begin{equation}
\int\limits_{0}^{1} \sqrt{\langle \dot q(t), \dot q(t) \rangle_{q(t)}} dt = \sqrt{2k} \ ,
\end{equation}
whereby $\delta = \sqrt{2k}$. Therefore, Eq.~\eqref{hamil} can be rewritten:
\begin{equation}
\delta = \sqrt{2k} \iff p^TI^{-1}(q)p = \delta^2.
\label{velocity}
\end{equation}
Taking equation~\eqref{velocity} at the initial state $t=0$, we can determine all of the initial velocities for which  $d_F(q(0), q(1)) = \delta$. These are needed to solve the system of ODEs~\eqref{hamileq} and compute the geodesics.

Generally, computing the geodesic between two distributions is challenging. Methods relying on shooting algorithms have been developed for this problem \citep{LeBrigant2017}. However, in our framework, we don't want to compute the distance between two given distributions, which implies to find the particular geodesic joining these, but we aim at screening a Fisher sphere as well as possible for a fixed distance (as detailed in section \ref{Section 4.2}).
In the Section 1 of the online supplementary material, we focus on numerical methods for computing geodesics by solving the system of ODEs~\eqref{hamileq}. We then illustrate these methods by computing Fisher spheres on the Gaussian manifold $\mathcal{S} = \{\mathcal{N}(\mu, \sigma^2), (\mu, \sigma) \in  \mathbb{R} \times \mathbb{R}_{+}^{*}\}$.

%%%%%%%%%%%%%%%%%%%%%%%%%%%%%%%%%%%%%%%%%%%%%%%%%%%%%%%%%%%%%%%%%
\section{Applying this to perturbed-law based indices}\label{Section 4}

The UQ robustness analysis detailed in \Cref{sec:intro} and \Cref{Section 2} aims to quantify the impact of a lack of knowledge about an input distribution on the UQ of model outputs. In \Cref{Section 3}, a coherent formal definition of density perturbation was proposed. 
We now illustrate how to use this solution to define a practical robustness analysis method. %construction on several applications.
First however, analyzing the effect of perturbing an input density  requires the definition of some criterion or ``index''  to summarize the effects on a QoI.

%%%%%%%%%%%%%%%
\subsection{Definition of perturbed-law based indices}\label{Section 4.1}

A PLI aims to measure the impact of the modification of an input density on a QoI, such as a quantile or a threshold exceedance probability of the model output \citep{lemser15,suebou16}.
In the following, we focus on a quantile of order $\alpha$, often used  as a risk measure in real-world applications  \citep{mouwil17,delsue18,lar19}. 

If we consider $\mathcal{X}$ to be a compact set of $\mathbb{R}^d$, then given the random vector $\textbf{X} = (X_1,\ldots,X_d) \in \mathcal{X}$ of our $d$ independent imperfectly known input variables, $G(\cdot)$ our numerical model, and $Y = G(\textbf{X}) \in \mathbb{R}$ the model output, the quantile of order $\alpha$ of $Y$ is given by:
\begin{equation}
q^{\alpha} = \inf\{t \in \mathbb{R}, F_{Y}(t) \geq \alpha\} \ ,
\end{equation}
where $F_Y$ is the cdf of the random variable $Y$. In order to compute the $i$-th PLI (see below), we first change the density $f_i$ of $X_i$ into a density $f_{i{\delta}}$, where $\delta \in \mathbb{R^+}$ represents the level of perturbation. The perturbed quantile is then given by:
\begin{equation}
q^{\alpha}_{i \delta} = \inf\{t \in \mathbb{R}, F_{Y, i \delta}(t) \geq \alpha\} \ ,
\end{equation}
where $F_{Y, i \delta}$ is the cdf corresponding to the input variable $X_i$ sampled from $f_{i\delta}$. The $i$-th PLI $S_{i}$ is then simply defined as the relative change in the output quantile generated by the perturbation, i.e.,
\begin{equation}
S_{i}(f_{i\delta}) = \frac{q_{i \delta}^{\alpha} - q^{\alpha}}{q^{\alpha}} \ .
\end{equation}
This definition slightly differs from that proposed in previous studies \citep{lemser15,sueioo17}, but is preferred here as it has been found to be more intuitive in a number of engineering studies that used PLI. Indeed, it simply consists of the relative variation of the quantile when the $i$-th input is submitted to a density perturbation and, as such, allows for a clearer interpretation. 

In many applications including nuclear safety exercises, computer models are costly in terms of CPU time and memory. This means that  
only a limited number $N$ of code runs are available for estimating all of the PLIs. 
Thus, we have a sample $\mathcal{Y}_N = \{y^{(n)}\}_{1 \leq n \leq N}$ of $N$ outputs of the model from a sample $\mathcal{X}_N = \{ \textbf{X}^{(n)} = (x_1^{(n)},...,x_d^{(n)})\}_{1 \leq n \leq N}$ of $N$ independent realizations of $\textbf{X}$. 
The estimation of the quantile is then based on the empirical quantile estimator  $\displaystyle \widehat{q}^{\alpha}_N = \inf\{ t \in \mathbb{R}, \widehat{F}_{Y}^{N}(t) \leq \alpha\}$, where $\displaystyle\widehat{F}_{Y}^N(t) = \frac{1}{N} \sum\limits_{n = 1}^N \mathds{1}_{(y^{(n)} \leq t)}$ is the empirical estimator of the cdf of $Y$. In order to estimate the perturbed quantile $\displaystyle\widehat{q}_{N, i\delta}^{\alpha}$ from the same sample $\mathcal{X}_N$, we use  reverse importance sampling \citet{hes96} (see Section 2 of the online supplementary materials) to compute $\displaystyle\widehat{F}_{Y,i\delta}^{N}$ \citep{delsue18}:
\begin{equation}
\widehat{F}_{Y,i\delta}^{N}(t) = \frac{\sum\limits_{n = 1}^N L_i^{(n)} \mathds{1}_{(y^{(n)} \leq t)}}{\sum\limits_{n = 1}^N L_i^{(n)}} \ ,
\end{equation}
with $L_i^{(n)}$ the likelihood ratio $\displaystyle\frac{f_{i\delta}(x_i^{(n)})}{f_i(x_i^{(n)})}$. The estimator of the PLI is then
\begin{equation}
\displaystyle \widehat{S}_{N, i\delta} = \frac{\widehat{q}_{N, i\delta}^{\alpha} - \widehat{q}_{N}^{\alpha}}{\widehat{q}_{N}^{\alpha}}\,.
\end{equation}
The theoretical properties of $\widehat{S}_{N, i\delta}$ are studied in Section 3 of the online supplementary materials.

As presented in \Cref{Section 3}, the Fisher sphere of radius $\delta$  centered on the initial input distribution $f_i$ and written $\partial\mathcal{B}_{F}(f_i, \delta) = \{g, \: d_F(f_i, g) = \delta\}$, provides a good base for perturbing distributions.
We do not consider one specific perturbation at level $\delta$, but instead a whole set of them: $\partial\mathcal{B}_{F}(f_i, \delta)$. Over this set, we compute the maximum $S_{i\delta}^{+}$ and the minimum $S_{i\delta}^{-}$ of the PLI:
\begin{equation}
S_{i\delta}^{+} = \max\limits_{g \in \partial\mathcal{B}_{F}(f_i, \delta)} S_{i}(g) \ ,
\end{equation}
\begin{equation}
S_{i\delta}^{-} = \min\limits_{g \in \partial\mathcal{B}_{F}(f_i, \delta)} S_{i}(g) \ ,
\end{equation}
where $S_i(g)$ is the PLI with $g$ the perturbed density for the variable $X_i$.

Among all perturbed distributions at level $\delta$, we look for those which most deviate the quantile from its original value. These two quantities $S_{i\delta}^{+}$ and $S_{i\delta}^{-}$, called OF-PLI (for ``Optimal Fisher-based PLI''), are measures of the robustness of the numerical code while taking into account uncertainty in the input distribution.

%%%%%
\subsection{Practical implementation}
\label{Section 4.2}

The main problem when estimating the OF-PLI arises from the available sample size being finite, and possibly not as large as it needs to be.
Indeed, due to inherent computational budget constraints in practice, there might not be enough sample points to correctly compute the perturbed quantile (and its confidence interval) at certain levels of perturbation.
Thus, the key issue is to determine the extent to which an input distribution should be perturbed.
To address this, we propose adapting the empirical criterion from \cite{iooss_robustness_2020} to help establish a maximal perturbed level $\delta_{max}$. For a proper estimation of OF-PLI indices, the number of points $N_\mathcal{Y}$ in the output sample $\mathcal{Y}_N$ exceeding the $\delta$-perturbed quantile needs to be sufficiently high.
As the sample size is finite, the $\delta$-perturbed quantile can often not be computed accurately enough for high values of $\delta$. In practice, a value of $N_\mathcal{Y}=10$ has been decided on (after numerous numerical tests) as the smallest allowed for computing a PLI-quantile. As soon as a distribution on the Fisher sphere exceeds the criterion in \citep{iooss_robustness_2020}, the corresponding radius is taken as $\delta_{max}$.

%The estimation of the quantity of interest $S_{i\delta}^{+}$ and $S_{i\delta}^{-}$ is summarized as follows:
%\begin{itemize}
%    \item Choose a level of perturbation $\delta$, an input number $i \in \modif{\{ 1,..., d\}}$ and a sample of $K$ points on the Fisher sphere of radius $\delta$ centered in $f_i$ using the numerical method of \Cref{Section 3.3}.
%    \item For each $\{f^{(k)}_{i\delta}\}_{1 \leq k \leq K}$ sampled on the Fisher sphere, estimate \modif{the $\alpha$-quantile for the $k$-th perturbed density} $q_{i \delta}^{\alpha, (k)}$ using the reverse importance sampling technique based on the sample $\mathcal{X}_N$.\modif{ Verify that the number of point in the output sample below or above the perturbed quantile $q_{i \delta}^{\alpha, (k)}$ satisfies the stopping criteria $N_{\mathcal{Y}}$.} Then, compute the PLI estimator $\widehat{S}^{(k)}_{N,i\delta}$.
%    \item The estimators $\widehat{S}_{N,i\delta}^{+}$ and $\widehat{S}_{N,i\delta}^{-}$ of the quantity of interest $S_{i\delta}^{+}$ and $S_{i\delta}^{-}$ are taken as the maximal and minimal value of the PLI sampled on the Fisher sphere $\{\widehat{S}_{N,i\delta}^{(k)}\}$.
%\end{itemize}

The estimation of the quantities of interest $S_{i\delta}^{+}$ and $S_{i\delta}^{-}$ is described in Algorithm \ref{alg:OF-PLI}.
We emphasize that this approach only applies to expensive computer models. Indeed, the bootstrap variance of the estimated quantile with reverse importance sampling tends to be very large, as illustrated in \cite{iooss_robustness_2020}.
This is due to the likelihood ratio, which explodes locally. Thus, when dealing with a cheap code, one can directly resample over the perturbed distribution in order to estimate the output quantile. In this situation, there is no limiting level of perturbation $\delta_{max}$.

\begin{algorithm}[!ht]
\caption{Estimation of $S_{i\delta}^{+}$ and $S_{i\delta}^{-}$}\label{alg:OF-PLI}
\begin{enumerate}
\item Initialisation: $\delta$ (level of perturbation), $i \in \{ 1,\ldots, d\}$ (input number), a sample of $K$ points on the Fisher sphere of radius $\delta$ centered at $f_i$ (using the numerical method of Section 2 of the online supplementary material).
\item For $k=1,\ldots,K$:
    \begin{enumerate}
    \item estimate the $\alpha$-quantile $q_{i \delta}^{\alpha, (k)}$ for the $k$-th perturbed density $\{f^{(k)}_{i\delta}\}$ using reverse importance sampling based on the sample $\mathcal{X}_N$,
    \item check whether the number of points in the output sample below or above the perturbed quantile $q_{i \delta}^{\alpha, (k)}$ satisfies the stopping criterion $N_{\mathcal{Y}} \geq 10$,
    \item compute the PLI estimator $\widehat{S}^{(k)}_{N,i\delta}$.
    \end{enumerate}
\item the estimators $\widehat{S}_{N,i\delta}^{+}$ and $\widehat{S}_{N,i\delta}^{-}$ of the quantities of interest $S_{i\delta}^{+}$ and $S_{i\delta}^{-}$ are taken as the maximal and minimal values of the PLI sampled on the Fisher sphere $\{\widehat{S}_{N,i\delta}^{(k)}\}$.
\end{enumerate}
\end{algorithm}

In future work, OF-PLI confidence intervals (computed via bootstrap) will provide valuable additional information like for instance confidence intervals, but further work on these is required and they are not included in the industrial application (section \ref{Section 5.2}).
The code for the OF-PLI is available at \url{ https://github.com/JeromeStenger/PLI-Technometrics}.
The code for computing the older version of the PLI, called E-PLI (entropy-based PLI) in the following, is available in the \texttt{sensitivity} package of the software \texttt{R}.

%%%%%%%%%%%%%%%%%%%%%%%%%%%%%%%%%%%%%%%%%%%%%%%%%%%%%%%%%%%
\section{Perturbed-law based indices in engineering studies}\label{Section 5}

The PLI as defined here allows us to assess to what extent an output quantile can be affected by an error of magnitude $\delta$ in the characterization of an input distribution. In the following section, we compare, using a toy example, our new method (OF-PLI) to the earlier one (E-PLI). 
Following that, we illustrate the use of OF-PLI in a nuclear safety analysis of a pressurized water nuclear reactor.
%More precisely, the PLI is applied in the UQ of a model which simulates a loss of coolant accident caused by an intermediate-size break in the primary loop (IB-LOCA).
% In this section, we investigate the behaviour of our two estimators $\widehat{S}_{N, i\delta}^{+}$ and $\widehat{S}_{N, i\delta}^{-}$ on numerical example. Two different models are presented. The first one is a simplfied model for the assessment of the risk of a flood on an industrial site and has a simple analytical expression. The second model is based on a system thermal-hydraulic code used for nuclear safety studies and is much more complex and computationally demanding.
Moreover, as PLIs are based on changes in input distributions, they differ from global sensitivity measures \citep{ioolem15}, which evaluate the effect of input variability for fixed probabilistic models. To study potential coherence and/or divergence between the two approaches, we compare Sobol indices and OF-PLI on an analytic  model in Section 4 of the online supplementary material.

%%%%%%%%%%%%%%%%%%
\subsection{A toy example: the Ishigami function}

The Ishigami function \citep{Ishigami} is often used as an example for uncertainty and sensitivity analysis methods, in particular because it exhibits strong non-linearity and non-monotonicity. In this section, we apply the methodology introduced in Section \ref{Section 4.2} to estimate the OF-PLI and compare our results to E-PLI. The Ishigami function, which takes three input random variables $(X_1, X_2, X_3)$, each distributed normally as $\mathcal{N}(0,1)$, is defined with the following analytic formula:
\begin{equation}
    G(x_1, x_2, x_3) = \sin(x_1)+7\sin(x_2)^2+0.1x_3^4\sin(x_1) \ .
\end{equation}

We intend to evaluate the impact of a perturbed input distribution on the value of the $95\%$-quantile.
In this simple example where the function is cheap to evaluate, we do not use the reverse importance sampling estimator of the quantile as proposed in Section \ref{Section 4.2}. Since the computational burden of running new code simulations is negligible, we instead draw new samples of size $N=2 000$ directly from the perturbed input distributions in order to compute the output perturbed quantile. 
Hence, this toy example does not follow the reverse important sampling estimation procedure detailed in Section \ref{Section 4.2}. The OF-PLI are computed for perturbation levels $\delta$ from the interval $[0, 0.9]$.  Notice that the choice $\delta_{max}=0.9$ is arbitrary here since the reverse importance sampling estimator has been put aside. Indeed, here there is no actual limit for the maximal perturbation level as the OF-PLI are computed by resampling from the perturbed distribution. 
%The PLI are computed using the methodology developed in Section \ref{Section 4} for $\delta$ varying in $[0, 0.95]$. Indeed, for $\delta>0.95$ the number of value above the perturbed quantile of the third input variable is lower than the required number $N_{\mathcal{Y}}=10$. 
We chose  a value of $K=100$ trajectories over each Fisher sphere for computing the minimum and maximum of the OF-PLI. 
% deltaMax=1 pour N=5
We also compute the $95\%$-confidence intervals calculated from $50$ values of $\widehat{S}^+_{N,i\delta}$ and $\widehat{S}^-_{N,i\delta}$.

The OF-PLI results are depicted in Figure \ref{fig: ishigami pli}.   It appears that the third input has the most impact in shifting the quantile to the right. On the other hand, the second input has more impact in shifting the quantile to the left. These results are consistent with the well-known behavior of the Ishigami function in terms both of non-linearity of the model and primary influence of the third input.

\begin{figure}[!ht]
    \centering
    \includegraphics[scale=0.95]{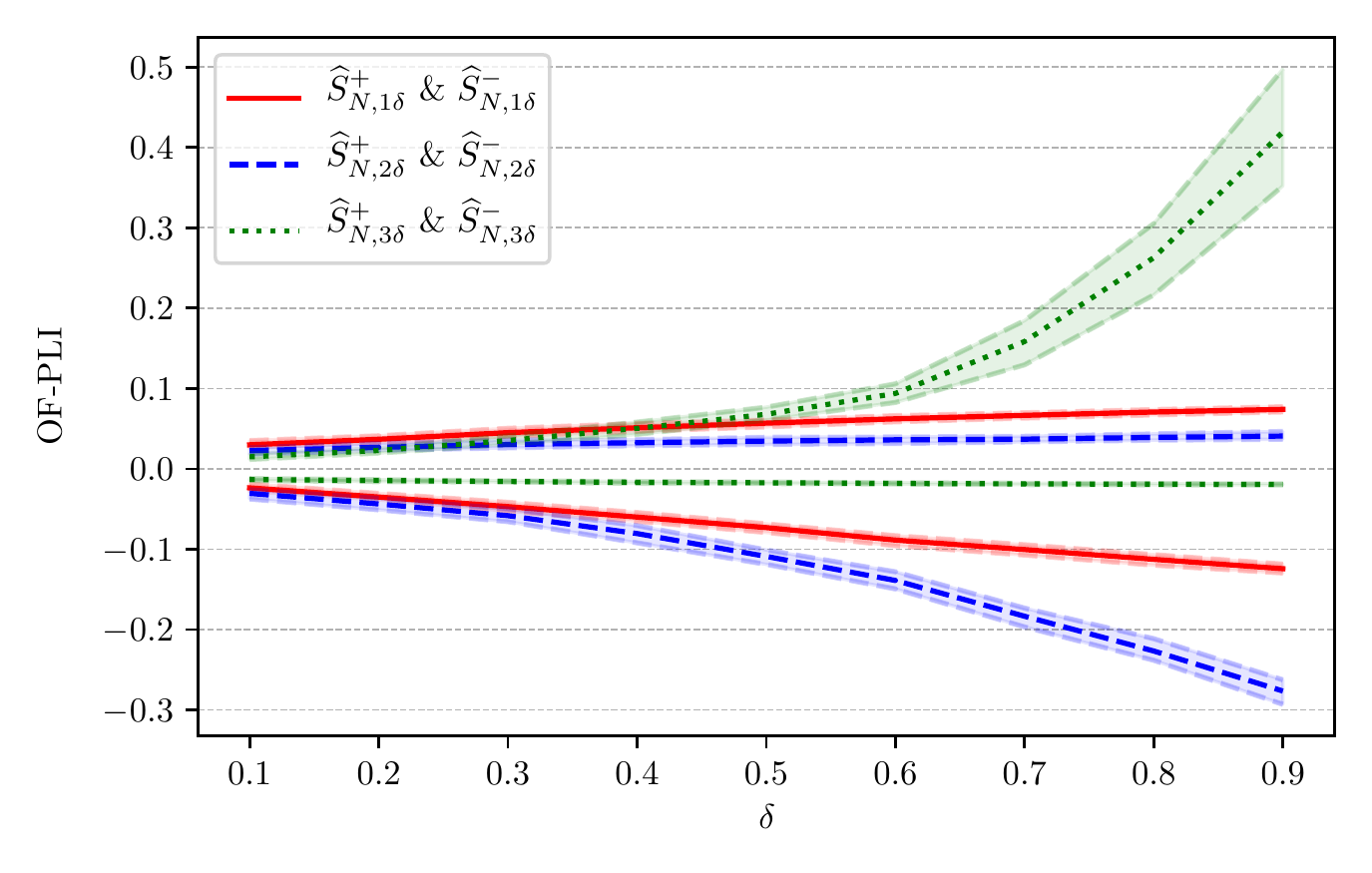}
    \caption{Minimum and maximum of the OF-PLI over the Fisher sphere with $K=100$ trajectories for $\delta$ varying in $[0,0.9]$, and their $95\%$-confidence intervals. For the Ishigami model, we directly resample from the perturbed distribution to compute the perturbed quantile instead of using the reverse important sampling estimator.}
    \label{fig: ishigami pli}
\end{figure}

Because the maximum and minimum of the OF-PLI are taken over the Fisher sphere, we depict in Figure~\ref{fig: ishigami fisher sphere} the distribution of the OF-PLI over the Fisher sphere with radius $\delta=0.9$ for the third input. 
One can see that in this case, the maximum and minimum are found respectively for high and low variance, with seemingly no change with respect to the mean.

\begin{figure}[!ht]
    \centering
    \includegraphics[scale=1]{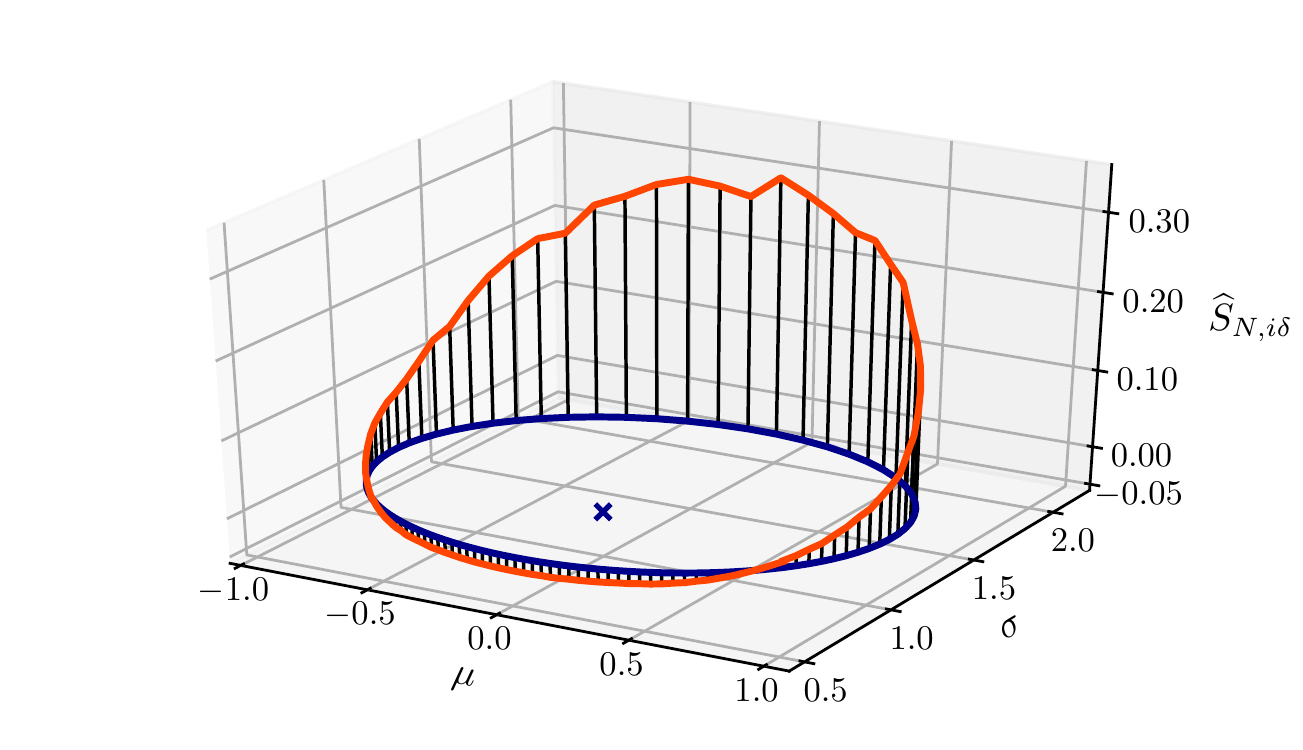}
    \caption{Value of the OF-PLI $\hat{S}_{N,i\delta}$ (red line) for the third input of the Ishigami model ($N=100$, $i=3$) over a Fisher sphere of radius $\delta=0.9$ (blue line).}
    \label{fig: ishigami fisher sphere}
\end{figure}
%\begin{figure}
%    \centering
%    \includegraphics{PLIvalue_Fisher_delta09_3emeInput.pdf}
%    \caption{Value of the PLI over a Fisher sphere at radius $\delta=0.9$ for the third input of the Ishigami model.}
%    \label{fig: ishigami fisher sphere}
%\end{figure}

These results can be compared to the E-PLI ones (see Section \ref{Section 2}). As in this case, inputs follow normal distributions, applying an inverse cdf transformation is not necessary. Therefore, perturbing the mean (resp. the variance) of the input variable is equivalent to drawing straight horizontal (resp. vertical) trajectories  in the parameter space (see Figure 1 of the online supplementary material). Results are depicted in Figure~\ref{fig: ancien pli}; the mean of the Gaussian is perturbed in $[-1, 1]$ and its variance in $[0, 4]$. 
This corresponds to the range of variation of these parameters for the Fisher sphere radius varying in $[0, 0.9]$. We then compare the third input for the two methods. We saw in Figure~\ref{fig: ishigami fisher sphere} that the maximal OF-PLI was obtained with high variance and no mean perturbation, which is coherent with the results in Figure~\ref{fig: ancien pli}. However, we miss the true impact a perturbed density can induce in situations where the maximal and minimal OF-PLI are not located in these two axes, such as for instance is the case for the first variable. Hence, the E-PLI, restricted to two directions in the Fisher sphere, has limited interpretation.  

\begin{figure}[!ht]
  \begin{subfigure}{.49\textwidth}
      \centering
      \includegraphics{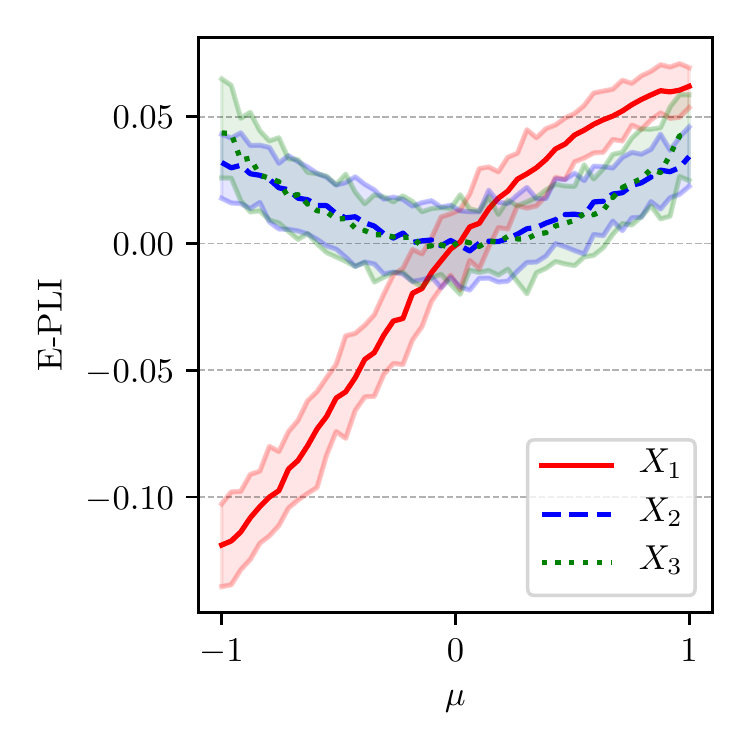}
  \end{subfigure}
  \begin{subfigure}{.49\textwidth}
      \centering
      \includegraphics{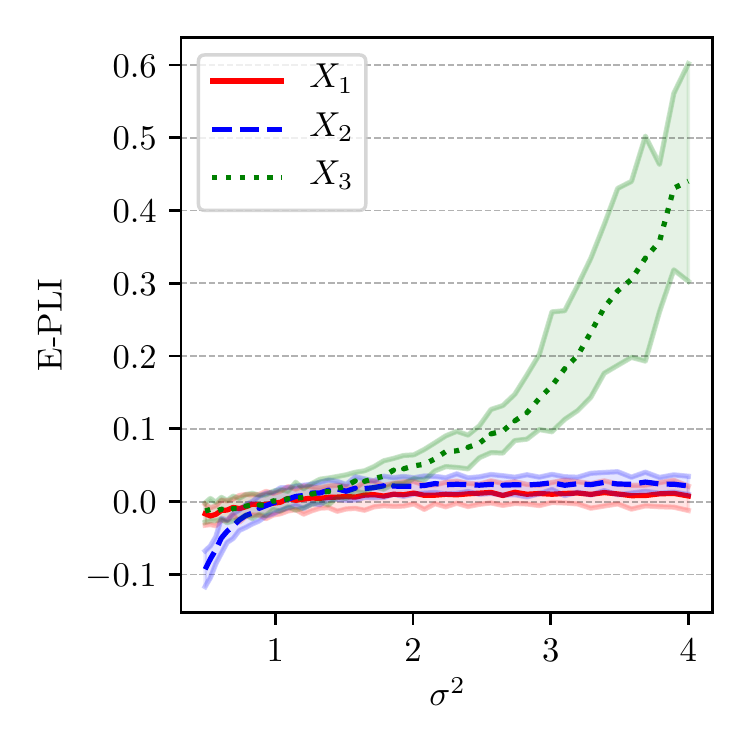}
  \end{subfigure}
  \caption{Computation of the E-PLI. Left: perturbation of the mean of the Gaussian distribution. Right: perturbation of the variance of the Gaussian distribution.}
  \label{fig: ancien pli}
\end{figure}

%%%%%%%%%%%%%%%%%%%%%%%%%%%%%%%%%%%%%
\subsection{A nuclear safety application}
\label{Section 5.2}

This industrial application concerns the study of the peak cladding temperature (PCT) of fuel rods in the case of a loss of coolant accident caused by an intermediate-sized break in the primary loop (IB-LOCA) in a nuclear pressurized water reactor. 
According to operating rules, this temperature must remain below a threshold to prevent any deterioration of the state of the reactor. The thermal-hydraulic transient caused by this accidental scenario is simulated with the CATHARE2 code \citep{gefant11}, providing a temperature profile throughout time for the surface of the nuclear core assemblies \citep{mazvac16}. The thermal-hydraulic model involves boundary and initial conditions, and many physical parameters (heat transfer coefficient, friction coefficient, etc.) whose exact values are unknown.
The probability distributions of these inputs can be obtained from data and expert knowledge, or by solving inverse problems on experimental data \citep{baczha19}.

Input uncertainties are propagated inside this model and the goal of UQ  is to estimate a high-order quantile of the PCT (model output).
This $\alpha$-quantile is interpreted as a pessimistic estimate of the PCT.
Like any scientific approach, this methodology is based on hypotheses; regulatory authorities require an evaluation of the effect of these hypotheses 
 on claimed results.  
Nuclear power operators are required to conduct studies in such a way to ensure that actual risks are overestimated. 
Under this ``conservative principle'' they are required to choose the most pessimistic assumption each time a modeling decision has to be made. 
In deterministic studies, this simply consists of taking the most penalizing value for each input variable. 
%It often also consists in replacing causal models built from experimental data by an "envelope" function covering all physical observations.
In this way, the resulting computation is seen as simulating the worst case scenario for the risk being examined. 
It is, however, not straightforward to implement such a principle when the numerical code is complex, with both non-monotonic effects of inputs and interactions between inputs. 
It is even harder to extend this rationale to a UQ framework that aims to represent all potential scenarios with related occurrence likelihoods.
Nevertheless, recent works \citep{lar19} have shown that the E-PLI can be useful as evidence in discussions on the choice of input distributions.

%,We now consider a real industrial use case related to nuclear safety analysis.
%We study the PLI on a numerical code called CATHARE. This code simulates the evolution of several physical variables during a thermal hydraulic transient \citep{mazvac16,ioomar19}.
In the present application, we study a reduced-scale mock-up of a pressurized water reactor with the $7$ imperfectly known inputs given in Table~\ref{tab: table2} \citep{delsue18}. 
To compute the OF-PLI, an input-output sample of size $N=1000$ is available, coming from a space filling design of experiments \citep{fanli06} (whose points in $[0,1]^d$ have been transformed to follow the input probability distributions).
More precisely, a Latin hypercube sample minimizing the $L^2$-centered discrepancy criterion \citep{jinche05} has been used. 
%This type of DoE provides a good coverage of the parameter space with a limited number of samples [une ref sur les SFD ?]. 
%It is hence suitable to UQ studies using costly computer codes, as well as the importance sampling of perturbed distributions allow to compute
The OF-PLI (with respect to a quantile of order $\alpha=0.95$) will then be estimated without any additional code run (see Section~\ref{Section 4.1}).
%during an accidental scenario called IBLOCA (Intermediate Break Loss Of Coolant Accident) which takes into account a specific piping rupture. Our variable of interest is the Peak Cladding Temperature (PCT) during the transient also called maximal heater rod temperature, this quantity is important in power plant design.

%\begin{figure}[!ht]
%            \centering
%            \includegraphics[scale = 0.5]{./REP_maquette.png}
%            \caption{Water pressure reactor, modeled by CATHARE2.}
%            \label{}
%\end{figure}

%The input variables of the system are physical parameters (heat transfer coefficient, friction coefficient,...) with associated probability density depicted in Table \ref{tab: table2}. For safety studies, one aims at quantifying the impact of the lack of information on the input distributions on an extreme quantile (95\% or 99\%) of the output variable. Hence, the PLI associated to the information geometry heuristic is well suited to perturbe the input densities. We dispose [fozami] of a dataset of $N=1000$ observations of the code sampled from a space filling design (Latin Hypercube Sampling) in order to provide a good coverage of the parameter space. No additional expensive code call are necessary to estimate the PLI as the perturbed distributions are sampled by importance sampling (see Section \ref{Section 4.1}).

\begin{table}[!ht]
    \centering
    \caption{Input variables of the CATHARE2 code with their associated probability distributions.}
%{\small
    \begin{tabular}{llr}
    \begin{tabular}{l}Variable\\number\end{tabular} & \begin{tabular}{l}Input\\name\end{tabular}&  Probability distribution \\
    \hline
    1 & STMFSCO & Uniform $\mathcal{U}([-44.9,63.5])$ \\
    2 & STBAEBU & Truncated Log Normal $\mathcal{LN}(0,0.76)$ on $[0.1, 10]$ \\
    3 & STOIBC1 & Truncated Log Normal $\mathcal{LN}(0,0.76)$ on $[0.1, 10]$ \\
    4 & STOIBC3 & Truncated Log Normal $\mathcal{LN}(0,0.76)$ on $[0.1, 10]$ \\
    5 & STOIDC & Truncated Log Normal $\mathcal{LN}(0,0.76)$ on $[0.1, 10]$ \\
    6 & STOICO & Truncated Log Normal $\mathcal{LN}(-0.1,0.45)$ on $[0.23, 3.45]$ \\
    7 & CLFBR & Truncated Normal $\mathcal{N}(6.4,4.27)$ on $[0,12.8]$ \\
    \end{tabular}
 %   }
    \label{tab: table2}
\end{table}

Figure~\ref{fig:cathare} presents the maximum and minimum values of our two estimators $\widehat{S}_{N, i\delta}^{+}$ and $\widehat{S}_{N, i\delta}^{-}$. We compute Fisher spheres with radius $\delta$ sampled uniformly in $[0.1, 0.5]$, all respectively centered on the initial input distributions. The maximal perturbation level $\delta_{max} = 0.5$ is derived from the stopping criterion introduced in Section~\ref{Section 4.2}. On every sphere, $K=100$ perturbed densities are sampled. 
The OF-PLIs are finally estimated on a dataset of size $1000$. The stopping criterion of Section~\ref{Section 4.2} gives a maximal admissible OF-PLI of $4\%$; this value is determined from the maximal admissible quantile for which there are $N_\mathcal{Y}=10$ sample points above it. In fact, we see that $\widehat{S}_{N, 7\delta}^{+}$ is close to this maximal admissible value. 

\begin{figure}[!ht]
    \centering
    \includegraphics[scale = 0.7]{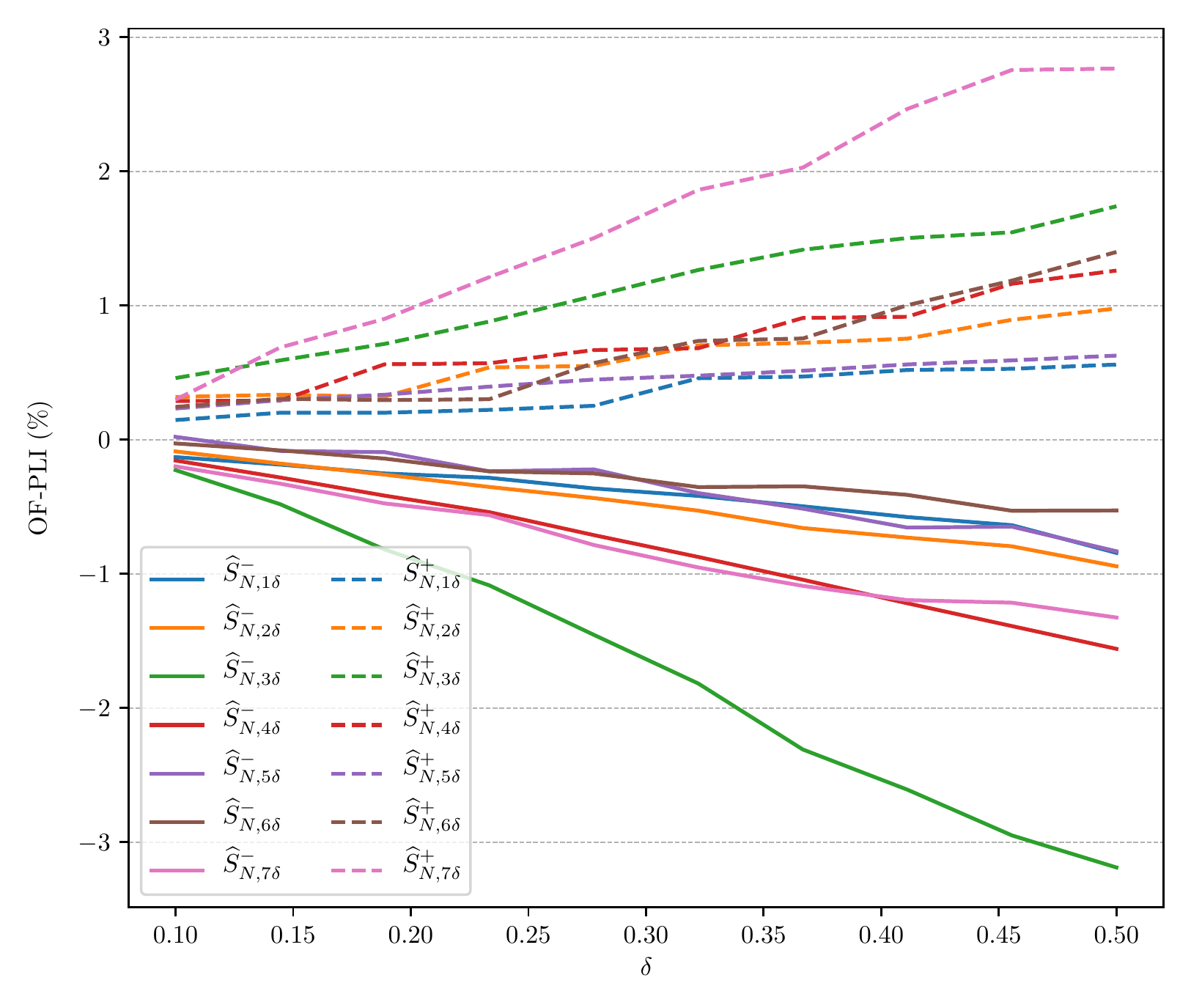}
    \caption{Bootstrap mean of the maximum and minimum of the OF-PLI $S_{i\delta}$ for the CATHARE2 code. Confidence intervals are not shown for the sake of clarity.}
    \label{fig:cathare}
\end{figure}

%We wish to compare our framework for computing the PLI based on the Fisher metric to the previous methodology based on the standard space transformation presented in Section \ref{Section 2.2}. We refer to \citet{sueioo17, delsue18} for details on that matter.
Studies conducted earlier for the same application \citep{delsue18} lead to similar results concerning the most influential inputs on the PCT's quantile: strong influence of variables $3$ and $4$, weak influence of variables $1$, $2$, and $5$. 
With respect to these previous studies based on the standard space transformation, our information geometry perturbation method evaluates the influence of variable $7$ as being smaller. 
In fact, as it is the only Gaussian distribution, the reverse transformation from the standard space to the physical one operates differently for this input than for the others. 
Overall, according to the values of $\widehat{S}_{N, 3\delta}^{+}$ and $\widehat{S}_{N, 7\delta}^{+}$, the variables $3$ and $7$ appear to be the most influential inputs on the value of the PCT's quantile. 
This was not observed using the standard space transformation, and is probably due to the fact that the standard space approach allows for perturbing only one of the probability distribution parameters (for example the expected value) at a time.
In contrast, our estimator corresponds to the maximal quantile deviation over an entire set of equivalent perturbations. This suggests two major advantages of our newly developed method: it (i) prevents interpretation bias induced by the standard space transformation, and (ii) allows for an exhaustive exploration of density perturbations for a given $\delta$.

%These result highlights the high impact on the quantile of the variables 3 and 4 for both methodologies, and oppositely the weak influences on the quantile of the variables 1, 2 and 5. However, the parameter 7 appears important using the standard space transformation while it is less influential with a perturbation based on the Fisher metric. Our hypothesis is that because it is the only Gaussian distribution, the reverse transformation of the random variable from the standard space to the physical space (Eq. (\ref{eq:rosenblatt})) operates differently that for the other input variables as discussed in Figure \ref{fig: kullback}. The last parameter (number $6$) is the one that increases the most the quantile in regards of $\widehat{S}_{N, 6\delta}^{+}$, this behavior was not observed with the standard space transformation. Indeed, the PLI computed with the standard space approach are based on only one perturbed density, while our estimator is the maximal quantile deviation over a whole set of perturbations. It may explain why some extreme behavior are not monitored [inap.] using the standard space transformation.

%%%%%%%%%%%%%%%%%
\section{Conclusion}

Based on the Fisher distance, we have defined an original methodology for perturbing input probability distributions in the specific case of mutually independent input random variables. 
The Fisher information is an intrinsic characteristic of probability measures and, in particular, does not depend on the specific choice of parametric representation. This fundamental property makes it the proper mathematical tool to compare perturbations on different imperfectly known physical inputs of a computer model, and also on different parameters of the same input distribution. It is even possible to remove all reference to a parametric sub-domain of the set of probability measures on $\mathcal{X}$, given the nonparametric extension of the Fisher distance proposed by \cite{holbrook17}. %{\color{red}[ref]}. 
However, the latter is limited by practical issues as it supposes a finite-dimensional representation of the densities, for example by means of projection onto an orthonormal basis of the probability space. This implies truncating the infinite sum of the projections of a given probability on all elements of the basis. This approximation will be poor for probabilities which are very different from those of the chosen basis. This shows that in practice it is not easy to eliminate the reference to a particular parametric model, even in a nonparametric framework.

Nevertheless, based on PLIs, our method  provides useful information on the most influential uncertainties related to the distributions of input variables, so-called ``epistemic uncertainties''. This is in particular crucial not only in making decisions concerning further research programs aiming at gaining better knowledge about these variables, but also to bring strong backing arguments to operators' safety demonstrations.
Indeed, we argue that this methodology is adequate for uncertainty studies with unreliable input distribution identification, or when an improved level of robustness is demanded for the choice of input distributions.
In our target application (nuclear licensing), the aim is not only to exhibit safety margin values for simulated accidents, but also to prove that the methodology as a whole does not induce any risk of underestimating these values. Hence we are not only looking for a worst case assessment method, but also for a more global understanding of how a potential error in an input's distribution affects the output.
In this sense, a practical option to increase conservatism in UQ studies is to replace one or several input distributions by penalized deterministic values, or by penalized versions of the distributions themselves. This nevertheless implies justifying the choice of the variables for which this penalization is done (see, e.g., \citet{largau20}).

Further investigations remain to be completed, due in part to this method increasing the numerical complexity and computational time (by several hours) required compared to the previous method of \citet{lemser15}.
Several Monte Carlo loops are needed to compute the maximal and minimal PLI over Fisher spheres and we have ongoing work on the improvement of the estimation of the maximum and the minimum of the PLI on a Fisher sphere. 
In particular, using a better-adapted programming language, a large computational gain (of hours) could be expected. 
There is a known numerical issue with the reverse importance sampling strategy whereby the likelihood ratio tends to explode, as do the confidence intervals. Moreover, the method consists in sampling trajectories over the Fisher sphere, but one could benefit from a more advanced strategy by directly optimizing the PLI over the sphere via, for instance, gradient descent along this manifold. 
The crucial problem of probabilistic dependency between inputs needs also to be studied in order to extend our framework to the non-independent input case; works in robustness analysis dealing with dependent inputs can be found for instance in \cite{Pesenti}.
Moreover, using a distance based on the Fisher information might be less intuitive for non-statisticians than simply shifting moments.
In order to promote our methodology in engineering applications, it is essential to explicit its link with statistical test theory, thus helping practitioners interpreting the results it provides.
Finally, future works will improve the estimation accuracy of the OF-PLI,  especially for high values of $\delta$ where numerical difficulties occurred, as mentioned in Section 1 of the online supplementary material.

%{\color{red}[+question d'un "niveau d'erreur acceptable"... i.e. d'une éventuelle controverse sur la valeur de $\delta$ à utiliser pour les démo de sûreté ???]}

%%%%%%%%%%%%%%%%%
%\section*{Acknowlegments}

%We are grateful to the useful and numerous comments of the two reviewers which greatly help to improve the manuscript. 
%We are grateful to Vincent Larget for helpful comments, as Kevin Bleakley, Chu Mai and Lynn Ferrieu for their help for the English writing. 

%%%%%%%%%%%%%%%%%%%%%%%%%%%%%%%%%%%%%%%%%%%%%%%%%%%%%%%%%%%%%%
\clearpage

  \begin{center}
    {\LARGE \bf Online Supplementary Materials for the Paper ``An Information Geometry Approach to Robustness Analysis for the Uncertainty Quantification of Computer Codes''}
\end{center}

%%%%%%%%%%%%%%%%%%%%%%%%%%%%
\section{Computing Fisher spheres: Numerical results}

Hamilton's equations (see Section 3.2 of the main article),
\begin{equation}
    \left \{
\begin{array}{rcccl}
    \dot q & = & \displaystyle \frac{\partial H}{\partial p} & = & I^{-1}(q)p \,,\\
    \dot p & = & - \displaystyle \frac{\partial H}{\partial q} & = & \displaystyle\frac{\partial L(t, q, I^{-1}(q)p)}{\partial q}\,,
\end{array}
\right.
\label{hamileqbis}
\end{equation}
can be solved using numerical approximation methods.
Figure \ref{fig: gaussgeodesics} illustrates our numerical solutions in the Gaussian case, i.e., when $\mathcal{S} = \{\mathcal{N}(\mu, \sigma^2), (\mu, \sigma) \in \mathbb{R} \times \mathbb{R_+^*}\}$. 
We compare the solution given by the explicit Euler algorithm to that given by the Adams-Moulton algorithm.
We recall that in the Gaussian case we have at our disposal an exact analytical expression of the Fisher sphere, as detailed in \cite{Costa}.
The Fisher sphere is centered at $\mathcal{N}(0, 1)$ with radius $\delta = 1$. Notice that there is no observable difference between the two methods in Figure \ref{fig: gaussgeodesics}. Hence, the numerical error is estimated from the variation of the Hamiltonian value that should be theoretically conserved along the geodesics. 
As a consequence, it is possible to quantify the performance of the numerical approximation by computing the value
$\displaystyle \Delta(t) = \frac{H(p(t), q(t)) - H(p(0), q(0))}{H(p(0), q(0))}$ for $t \in [0, 1]$. Here, $\Delta$ represents the relative variation of the Hamiltonian along the path $q$ computed with our numerical methods. 

\begin{figure}[!ht]
            \centering
            \includegraphics[scale = 1]{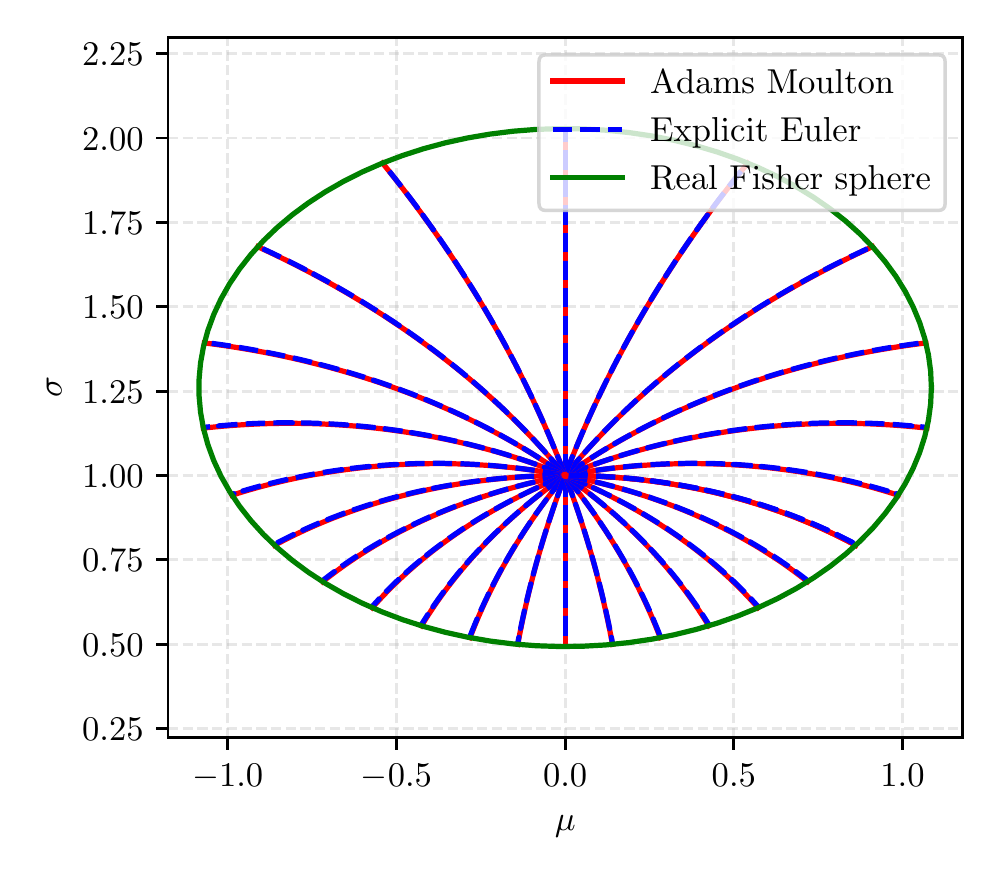}
            \caption{Geodesics in Gaussian information geometry computed with Euler's explicit method and the Adams-Moulton method. The radius $\delta$ is equal to 1.}
            \label{fig: gaussgeodesics}
\end{figure}

Figure \ref{fig: Delta} displays the value of $\Delta(t)$ for $t \in [0, 1]$ for one arbitrary geodesic from Figure \ref{fig: gaussgeodesics}.
The relative error for the Adams-Moulton method is negligible, while the maximum relative error for the explicit Euler scheme is around 0.3\%.
Hence, in the Gaussian case the Adams-Moulton scheme is preferable. 
However, instabilities have been observed in practice.
Symplectic methods  \citep{amanag00,Leimkuhler}, and in particular a symplectic Euler algorithm, could help to alleviate this problem by forcing the Hamiltonian to be constant. This will be the subject of  future work. 
Note also that truncation can lead to other numerical errors when the radius $\delta$ is too large.
Indeed, the normalization constant of some truncated distributions can sometimes become smaller than a computer's machine precision. 

\begin{figure}[!ht]
            \centering
            \includegraphics[scale = 1]{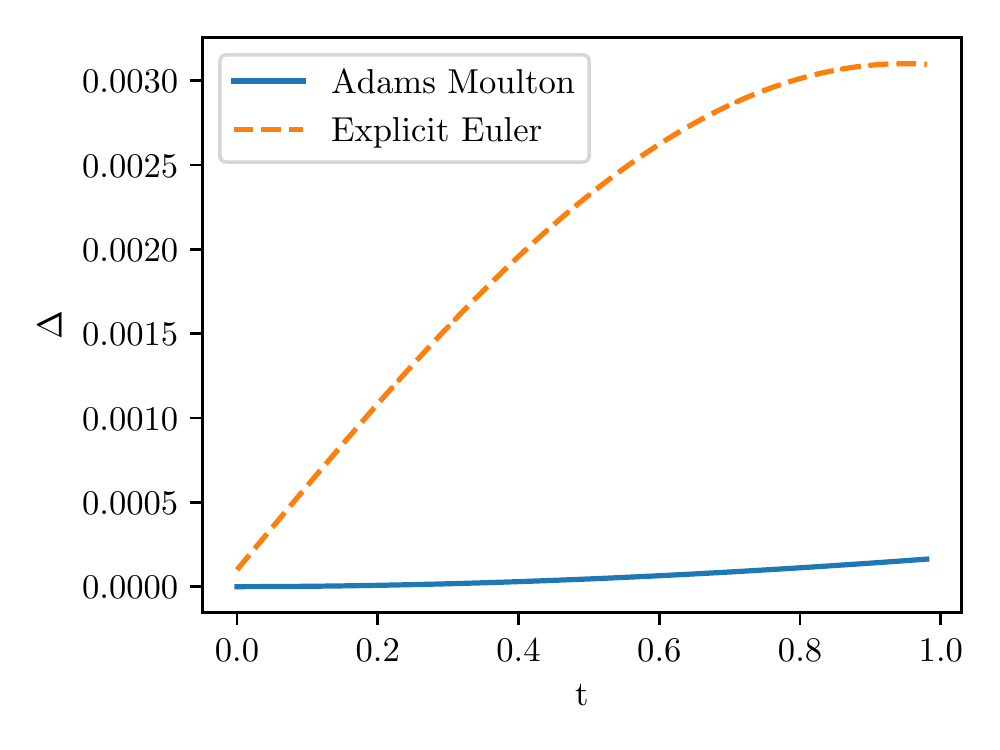}
            \caption{Relative variation of the Hamiltonian $\Delta$ along a geodesic for two different numerical methods.}
            \label{fig: Delta}
\end{figure}

%%%%%%%%%%%%%%%%%%
\section{The reverse importance sampling method}

Importance sampling (IS) is a Monte Carlo technique to compute statistical quantities such as $\mathbb{E}[\phi(X)]$ ($\textbf{X} = (X_1,...,X_d)$) with a
sample $\tilde{\mathcal{X}}_N = \{ \tilde{{X}}^{(1)},\ldots,\tilde{{X}}^{(N)} \}$
following a different distribution from that of $X$.
If we call $\pi$ the density of $X$ and $\tilde{\pi}$ the sampling one, we have
$$\mathbb{E}_\pi[\phi(X)]=\int\phi(x)\pi(x)dx = \int\phi(x)\tilde{\pi}(x) \frac{\pi(x)}{\tilde{\pi}(x)} dx  = \mathbb{E}_{\tilde{\pi}}[\phi(X)L(X)] ,$$
where $L={\pi}/{\tilde{\pi}}$ is the likelihood ratio allowing us to convert an integral under the distribution of $X$ into an integral under the sampling one. This leads to estimate of $M=\mathbb{E}[\phi(X)]$ via:
\begin{equation}
\hat{M}_{IS}=\frac{1}{N} \sum_{n=1}^N \phi(\tilde{X}^{(n)})L(\tilde{X}^{(n)}),
\label{eq:is_est}
\end{equation}
with an estimation error given by $\mbox{Var}[\hat{M}_{IS}]=N^{-1} \mbox{Var}_{\tilde{\pi}}[\phi L]$.

The use of an alternative density (called an \textit{instrumental density}) is most often
a variance reduction method \citep{kah53},
as the variance of $\hat{M}$ can be much lower than that of the usual empirical mean for a well chosen $\tilde{\pi}$. Indeed, an optimal variance reduction can be obtained for a sampling distribution $\propto \pi\phi$, which however implies that we know the value of  $\mathbb{E}[\phi]$ we are aiming to estimate.

The principle is similar if we aim to estimate quantiles corresponding to a perturbed
density, though from a sample generated from the initial density
of $X$ \citep{hes96}.
Here the IS is called ``reverse IS'' as we do not change the sampling density to achieve a better estimation of the initial quantity, but instead change the quantity to be estimated with the same
initial sample.
%As our purpose is to estimate a quantile of order $\alpha$ of $Y=G(X)$, we use an
%IS-estimator of the cdf of $Y$ to plug into the definition of the quantile
%$q^{\alpha} = \inf\{t \in \mathbb{R}, F_{Y}(t) \geq \alpha\} $.
Our purpose is to estimate a perturbed quantile of $Y$, e.g., a quantile of the distribution of $Y_{i\delta}=G(X_1,\ldots,X_{i-1},\tilde{X}_{i},X_{i+1},\ldots,X_d)$, where
the $i$-th input $\tilde{X}_{i}$ is perturbed and has a density of $f_{i\delta}$
instead of the initial one.

Hence the initial density of $X$ can be seen as the instrumental density of a typical IS-Monte Carlo scheme, and the perturbed density is the one under which we must estimate the cdf
$F_{Y,i\delta}$. 
If we denote $\pi(x_1,\ldots,x_d)=\Pi_{i=1}^d f_i(x_i)$ the initial density and
$\tilde{\pi}(x_1,\ldots,x_d)=f_1(x_1)\times...\times f_{i-1}(x_{i-1})\times f_{i\delta}(x_i)\times f_{i+1}(x_{i+1})\times\ldots\times f_d(x_d)$ the perturbed one as regards the
$i$-th input, and recalling that $F_{Y,i\delta}(y)=\mathbb{E}_{\tilde{\pi}}[\mathds{1}_{(G(X)\leq y)}]$, we have the following estimator for the perturbed cdf: 
\begin{equation}
\widehat{F}_{Y,i\delta}^{N}(y) = \frac{1}{N} \sum\limits_{n = 1}^N L_i^{(n)} \mathds{1}_{(y^{(n)} \leq y)}.
\label{eq:is_cdf}
\end{equation}
In this estimator, illustrated in Figure \ref{fig: Empirical_CDF_ini_perturbed_IS}, we recognize that of 
formula~(\ref{eq:is_est})
with $\phi(\cdot)=\mathds{1}_{(G(\mathbf \cdot)\leq y)}$ and $L(x^{(n)}_1,\ldots,x^{(n)}_d)={\pi(x^{(n)}_1,\ldots,x^{(n)}_d)}/{\tilde{\pi}(x^{(n)}_1,\ldots,x^{(n)}_d)}=f_i(x^{(n)}_i)/f_{i\delta}(x^{(n)}_i)=L_i^{(n)}$.

\begin{figure}[!ht]
    \centering
    \includegraphics[scale=0.95]{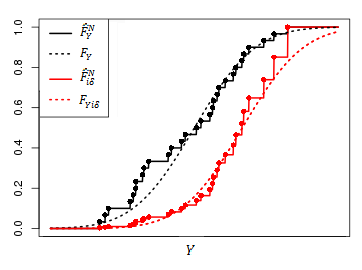}
    \caption{Empirical (solid lines) and theoretical (dotted lines) cumulative distribution functions. The samples (at the points' locations) used to estimate the initial distribution function $F_{Y}$ (black lines) and the perturbed one $F_{Y,i\delta}$ (red lines) are the same. Steps are all equal to $1/N$ for the initial cumulative distribution function, whereas they are weighted by the likelihood ratio in $\widehat{F}_{Y,i\delta}^{N}$.}
    \label{fig: Empirical_CDF_ini_perturbed_IS}
\end{figure}

The estimator $\displaystyle\widehat{q}_{N, i\delta}^{\alpha}$ of the $\alpha$-order quantile of $Y_{i\delta}$ is then obtained by plugging $\widehat{F}_{Y,i\delta}^{N}$ into the definition of
the theoretical quantile $q^{\alpha}_{i\delta}=\inf\{t \in \mathbb{R}, F_{Y,i\delta}(t) \geq \alpha\}$,
which gives:
$$\displaystyle\widehat{q}_{N, i\delta}^{\alpha}=\inf\{t \in \mathbb{R}, \widehat{F}_{Y,i\delta}^{N}(t) \geq \alpha\}.$$

Note that the estimator (\ref{eq:is_cdf}) of $F_{Y,i\delta}$ is normalized by $1/N$, which slightly differs from the one presented here, which is normalized by $1/\sum\limits_{n = 1}^N L_i^{(n)}$,
as well as the one plotted in Figure~\ref{fig: Empirical_CDF_ini_perturbed_IS}.
This guarantees $\widehat{F}_{Y,i\delta}^{N}(y)\rightarrow1$ when $y\rightarrow +\infty$ for finite values of $N$
without changing the asymptotic properties of the quantile estimator \citep{glynn_importance_nodate}.

%%%%%%%%%%%%%%%%%%
\section{Theoretical properties of the PLI-quantile estimator}

In this section, we investigate some theoretical aspects of the PLI estimator $\widehat{S}_{N, i\delta}$. 
As it is based on quantile estimators, we first focus on the 
asymptotic properties of the estimator $\left(\widehat{q}^{\alpha}_{N},  \widehat{q}^{\alpha}_{N,i \delta}\right)$.

\begin{theorem}\label{th: quantile}
Suppose that $F_Y$ is differentiable at $q^{\alpha} = F_Y^{-1}(\alpha)$ with $F'_Y(q^{\alpha}) > 0$ and that $F_{Y, i\delta}$ is differentiable at $q^{\alpha}_{i\delta} = F^{-1}_{Y, i\delta}(\alpha)$ with $F'_{Y, i\delta}(q^{\alpha}_{i\delta}) > 0$. We define:
\begin{equation}
\sigma_i^2 = \frac{\alpha(1-\alpha)}{F_Y'(q^{\alpha})^2} \; ,
\end{equation}
\begin{equation*}
\tilde{\sigma}_{i\delta}^2 = \frac{\mathbb{E}\left[\left(\frac{f_{i\delta}(X_i)}{f_i(X_i)}\right)^2(\mathds{1}_{(G(\textbf{X}) \leq q^{\alpha}_{i\delta})} - \alpha)^2\right]}{F'_{Y, i\delta}(q^{\alpha}_{i\delta})^2} \; ,
\end{equation*}
\begin{equation*}
\tilde{\theta}_i = \frac{\mathbb{E}\left[\frac{f_{i\delta}(X_i)}{f_i(X_i)} \mathds{1}_{(G(\textbf{X}) \leq q^{\alpha})}\mathds{1}_{(G(\textbf{X}) \leq q^{\alpha}_{i\delta})}\right] - \alpha\mathbb{E}[\mathds{1}_{(G(\textbf{X}) \leq q^{\alpha}_{i\delta})}]}{F'_Y(q^{\alpha})F'_{Y, i\delta}(q^{\alpha}_{i\delta})} \; .
\end{equation*}
Let us write the covariance matrix $\Sigma$ as:
$$\Sigma = \begin{pmatrix} 
\sigma^2 & \tilde{\theta}_i \\
\tilde{\theta}_i & \tilde{\sigma}_{i\delta}^2
\end{pmatrix},$$ and suppose that it is invertible and $\displaystyle\mathbb{E}\left[\left(\frac{f_{i\delta}(X_i)}{f_i(X_i)}\right)^3\right] < +\infty$. Then
\begin{equation*}
\sqrt{N}\left(\begin{pmatrix}
\widehat{q}^{\alpha}_N \\
\widehat{q}^{\alpha}_{N, i \delta}
\end{pmatrix} - \begin{pmatrix}
q^{\alpha} \\
q^{\alpha}_{i \delta}
\end{pmatrix}\right) \xrightarrow{\mathcal{L}} \mathcal{N}(0, \Sigma) \; .
\end{equation*}
\end{theorem}

\begin{proof}
We study the consistency and asymptotic normality of specific $M$ and $Z$-estimators in order to establish the proof. 
We suppose the theory around these estimators is known so that details can be kept to the bare minimum.
Further details can be found in Chapters 5.2 and 5.3 of \citet{vaart_asymptotic_2000}. Given a sample $(\mathbf{X}^{(n)})_{n \in (1,\ldots,N)}$ where $\mathbf{X}$ is a $d$-dimensional random vector, we define
\begin{equation}
\begin{array}{rcl}
\eta & = & \displaystyle\frac{\alpha}{1 - \alpha} \ ,\\
m_{\theta}(x) & = & \displaystyle -(G(x) - \theta)\mathds{1}_{(G(x) \leq \theta)} + \eta(G(x) -\theta)\mathds{1}_{(G(x) > \theta)} \ , \\
M_N(\theta_1, \theta_2) & = & \displaystyle\frac{1}{N} \sum\limits_{n=1}^N m_{\theta_1}(\mathbf{X}^{(n)}) + \frac{f_{i\delta}(X_i^{(n)})}{f_i(X_i^{(n)})}m_{\theta_2}(\mathbf{X}^{(n)}) \ , \\
\hat{\mathbf{\theta}}_N & = & \argmax M_N(\theta_1, \theta_2) \ .
\end{array}
\end{equation}
$\hat{\theta}_N$ is defined such that its two components correspond respectively to the estimators $\hat{q}^{\alpha}_N$ and $\hat{q}_{N,i\delta}^{\alpha}$ of the quantile and the perturbed quantile. The map $\theta \mapsto \nabla_{\mathbf{\theta}}M_N(\mathbf{\theta})$ with $\mathbf{\theta} = (\theta_1, \theta_2)^T$ has two non-decreasing components (it is a sum of non-decreasing maps). Now, by definition of $\hat{\theta}_N$ and concavity of $M_n(\theta)$, it holds that $\nabla_{\mathbf{\theta}}M_N(\hat{\mathbf{\theta}}_N) = 0$. Furthermore, we have that  $\displaystyle\nabla_{\mathbf{\theta}}M_N(\mathbf{\theta}) \xrightarrow{P} [(1+\eta)F_Y(\theta_1) - \eta, ((1+\eta)F_{Y, i\delta}(\theta_2) - \bar{L}_N\eta)]^T$ with $\displaystyle\bar{L}_N = \frac{1}{N} \sum\limits_{n=1}^N \frac{f_{i\delta}(X_i^{(n)})}{f_i(X_i^{(n)})}$, and this limit is a strictly non-decreasing function. Therefore, the assumptions of Lemma 5.10 in \cite[p.47]{vaart_asymptotic_2000} are satisfied, proving the consistency of the estimator $\hat{\theta}_N \xrightarrow{P} (q^{\alpha}, q_{i\delta}^{\alpha})^T$.

The asymptotic normality is studied via the map $\displaystyle \bar{m}_{\mathbf{\theta}}(x) \mapsto m_{\theta_1}(x) + \frac{f_{i\delta}(x)}{f_i(x)}m_{\theta_2}(x)$ which is Lipschitz for the variable $\mathbf{\theta}$ with Lipschitz constant $\displaystyle h(x) = \max(1, \eta)\left(1 + \frac{f_{i\delta}(x_i)}{f_i(x_i)}\right)$. The function $h$ belongs in $L^2$ if $\displaystyle \mathbb{E}\left[\left(\frac{f_{i\delta}(X_i)}{f_i(X_i)}\right)^2\right] < +\infty$. The map $\bar{m}_{\mathbf{\theta}}$ is also differentiable at $\displaystyle \mathbf{\theta}_0 = \argmax_{\mathbf{\theta} \in \Theta} \mathbb{E}[\bar{m}_{\mathbf{\theta}}(X)]$ with gradient:

\begin{equation}
\nabla_{\mathbf{\theta}_0} \bar{m}_{\mathbf{\theta}_0}(x) = ((1 + \eta)\mathds{1}_{(G(\mathbf{x}) \leq \theta_1)} - \eta, \: \frac{f_{i\delta}(x_i)}{f_i(x_i)}((1+\eta)\mathds{1}_{(G(\mathbf{x}) \leq \theta_2)} - \eta))^T \ .
\end{equation}                                            
Moreover, the map $\mathbf{\theta} \rightarrow \mathbb{E}[\bar{m}_{\mathbf{\theta}}(\mathbf{X})]$ has the following Hessian:
\begin{equation}
V_{\mathbf{\theta}_0} = \begin{pmatrix}
(1+\eta)F'_Y(q^{\alpha}) & 0 \\
0 & (1+\eta)F'_{Y, i\delta}(q^{\alpha}_{i\delta})
\end{pmatrix} \ ,
\end{equation}                                            
which is non-negative definite symmetric  whenever $F'_Y(q^{\alpha}) > 0$ and $F'_{Y, i\delta}(q^{\alpha}_{i\delta}) > 0$. Hence, Theorem 5.23 in \cite[p.53]{vaart_asymptotic_2000} applies.
This proves the asymptotic normality of the estimator $(\hat{q}^{\alpha}, \hat{q}^{\alpha}_{i\delta})^T$.
\end{proof}

The PLI $S_{i\delta}$ is a straightforward transformation of $\left(q^{\alpha}, q^{\alpha}_{i \delta}\right)^T$. 
To obtain the almost sure convergence of $\widehat{S}_{N, i\delta}$ to $S_{i\delta}$, it suffices to apply the continuous mapping theorem to the function $\displaystyle s(x,y) = \frac{y - x}{x}$. 

\begin{theorem}\label{th: pli}
Under the assumptions of theorem \ref{th: quantile}, we have
\begin{equation}
\sqrt{N}(\widehat{S}_{N,i\delta} - S_{i\delta}) \xrightarrow{\mathcal{L}} \mathcal{N}(0, d_s^T\Sigma d_s), \mbox{ with } d_s = \begin{pmatrix}
-q^{\alpha} / {q^{\alpha}_{i\delta}}^2 \\
1 / q^{\alpha}
\end{pmatrix}.
\end{equation}

\end{theorem}
Notice that the asymptotic variance relies on the initial $\alpha$-quantile and the perturbed quantile, which are precisely what we want to estimate. Hence, Theorem \ref{th: pli} cannot be used for building asymptotic confidence intervals.
However, its convergence properties are important for our method's credibility and acceptance.
In practice, the estimation error can be measured using bootstrapping \citep{Efron1979}.

%%%%%%%%%%%%%%%%%%%%%%%%%%%%%%%%%%%%%
\section{Applying PLI to an analytic flood risk model}

The model of interest concerns a flooded river simulation, which is especially useful in assessing the risk of submergence of a dike protecting industrial sites near a river. 
To this end, we use a model implementing a simplified version of the 1D hydro-dynamic Saint-Venant equations. This model computes $H$, the maximal annual water level of the river, from four parameters $Q$, $K_s$, $Z_m$ and $Z_v$, which are considered imperfectly known:
% We address a simplified hydraulic model that computes the maximal annual water height $H$ of a river subject to flood event. The computer code takes 4 input variables representing physical parameters. The inputs are modeled as random variables with associated truncated distributions depicted in Table \ref{tab: table1} \citep{ioolem15}.
% The model is based on the 1D hydro-dynamical equations of Saint Venant, its simplified analytical expression is given as follow:
\begin{equation}
    H = \left(\frac{Q}{300K_s\sqrt{2.10^{-4}(Z_m - Z_v)}}\right)^{0.6} \ .
    \label{eq: flood}
\end{equation}
The inputs are modeled as random variables with the associated truncated distributions given in Table \ref{tab: table1} \citep{ioolem15}.

\begin{table}[!ht]
    \centering
    \caption{Input variables of the flood model with their associated probability distributions.}
    \begin{tabular}{lrrrr}
    Input n$^\circ$ & Name & Description &  Probability distribution & Truncation  \\
    \hline
    1 & $Q$ &Maximal annual flowrate & Gumbel $\mathcal{G}(1013, 558)$ & $[500, 3000]$ \\
    2 & $K_s$ & Strickler coefficient & Normal $\mathcal{N}(30, 7.5)$ & $[15, +\infty]$ \\
    3 & $Z_v$ & River downstream level & Triangular $\mathcal{T}(50)$ & $[49, 51]$ \\
    4 & $Z_m$ & River upstream level & Triangular $\mathcal{T}(55)$ & $[54, 56]$
    \end{tabular}
    \label{tab: table1}
\end{table}

In global sensitivity analyses, Sobol indices are the most popular sensitivity measures because they are easy to interpret; each Sobol  index represents a share of the output variance, and all indices sum to $1$ under the assumption of independent inputs \citep{SOBOL2001,saltar02,pritar17}. 
We will compare these to the results of our robustness analysis framework in order to illustrate differences.
Note however that these conventional Sobol indices focus on the central part of the distribution (variance of the output).
We then also compute the target Sobol indices \citep{marcha20}, i.e., Sobol indices applied to the indicator function of exceeding a given threshold (chosen here as the $95\%$-quantile of the output).
To compute the first order and total Sobol indices of the inputs of the flood model (Eq.~(\ref{eq: flood})), the asymptotically efficient pick-freeze estimator \citep{pritar17} is used with an elementary Monte Carlo matrix of size $10^6$.
This gives a total cost of $N=6\times10^6$ model runs and a standard deviation of the indices' estimation error smaller than $10^{-3}$.
As shown in Table \ref{tab: sobol}, in the central part of the distribution (conventional Sobol indices),  we observe that the variable $Q$ is clearly more influential than the variable $K_s$, whereas $Z_v$ and $Z_m$ appear to have almost no influence on the output.
From the target Sobol indices, we observe that in the extreme part of the distribution (close to the $95\%$-quantile), $Q$ and $K_s$ have the same total effect 
(due to a strong interaction effect between them, which comes from the fact that threshold exceedence is principally sensitive to the joint values of these two inputs). 

\begin{table}[!ht]
    \centering
    \caption{Sobol index estimates of the flood model inputs.}
    \begin{tabular}{lrrrr}
    Inputs  & $Q$ &  $K_s$ & $Z_v$ & $Z_m$  \\
    \hline
    First-order Sobol indices & 0.713 & 0.254 & 0.006 & 0.006 \\
    Total Sobol indices & 0.731 & 0.271 & 0.008 & 0.008 \\
   First-order target Sobol indices & 0.242 & 0.125 & 0.002 & 0.002 \\
   Total target Sobol indices & 0.867 & 0.739 & 0.119 & 0.121
    \end{tabular}
    \label{tab: sobol}
\end{table}

We compute the OF-PLI (w.r.t. a quantile of order $\alpha=0.95$) for the flood model inputs by increasing the Fisher spheres radii $\delta \in [0, 1.4]$ by steps of $0.1$. The spheres are respectively centered on the distributions shown in Table~\ref{tab: table1}. On each of these spheres, we compute the OF-PLI for $K=100$ different perturbed distributions using a sample of $N=2000$ points distributed according to the initial distribution. The maximal radius $\delta_{max}=1.4$ comes from the stopping criteria explained in Section 4.2 (``Practical implementation'') of the paper. More precisely, the criterion is hit for the first input $Q$ at perturbation level $\delta>1.4$, meaning there are less than $N_\mathcal{Y}=10$ sample points above the maximal perturbed quantile. 

Figure~\ref{fig: gumbel} depicts how the Fisher sphere centered at the variable $Q$ deforms, and how the perturbed densities are spread out around the initial distribution. 
Figures~\ref{fig: distrib cas crue 01} and~\ref{fig: distrib cas crue 14} indicate that the maximal values of the OF-PLI are obtained  for densities allocating the largest probabilities to large input values (blue curves), whereas the minimal values are obtained for densities with the least weighted tails (green curves).
This behavior was predictable here since the river height is an increasing function of river flow (see Eq.~\eqref{eq: flood}), though this type of analysis can in practice provide substantial information in real-world engineering studies. Finally, we also observe that Fisher spheres sometimes have nontrivial shapes, as illustrated in Fig.~\ref{fig: gumbel sphere}. In this plot, we see that Fisher spheres are non-closed curves for radius $0.3\leq\delta$, the parameter's domain being bounded by $0$ for $\beta$ and $\gamma$. Such behaviour is specific to each parametric family of densities and, for instance, is not observed for non-truncated normal distributions. 

\begin{figure}[!ht]
  \begin{subfigure}{\textwidth}
    \centering
    \includegraphics{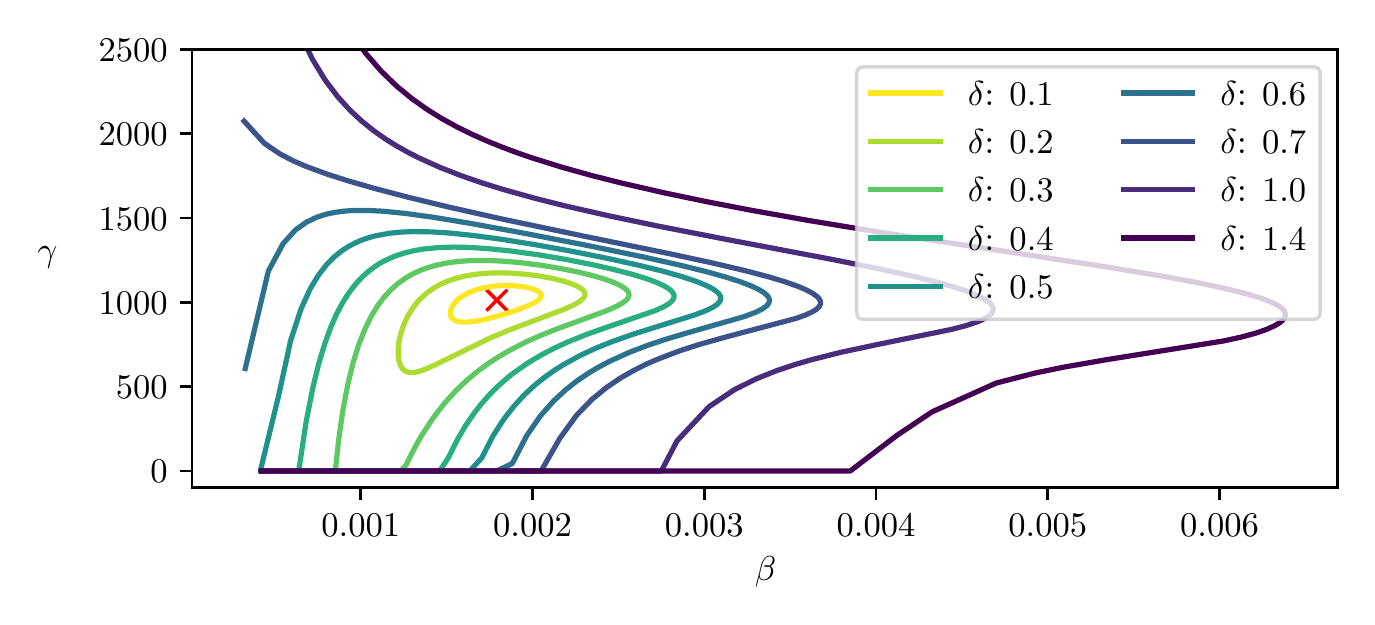}
    \caption{Deformation of the Fisher sphere for increasing radius $\delta$.}
    \label{fig: gumbel sphere}
    \end{subfigure}
  \begin{subfigure}{.48\textwidth}
      \centering
      \includegraphics{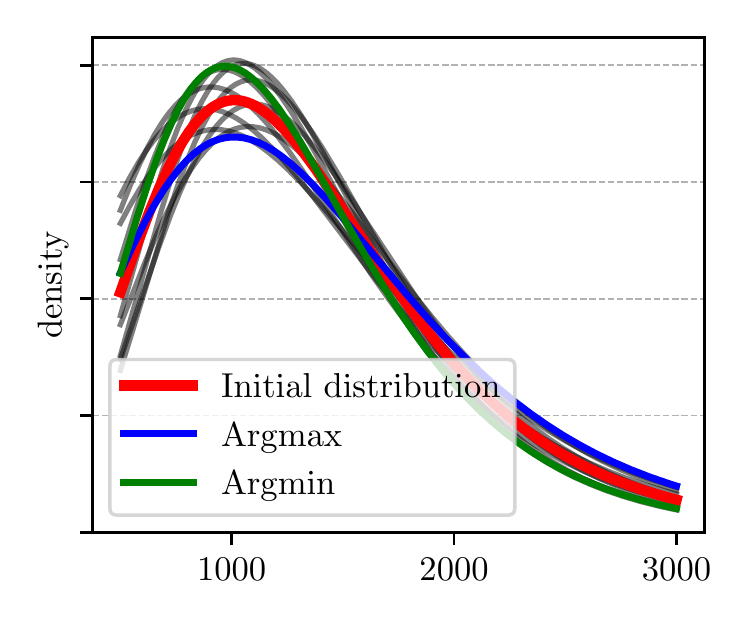}\caption{Densities over the Fisher sphere ($\delta = 0.1$). }
      \label{fig: distrib cas crue 01}
  \end{subfigure}
  \hfill
  \begin{subfigure}{.48\textwidth}
      \centering
      \includegraphics{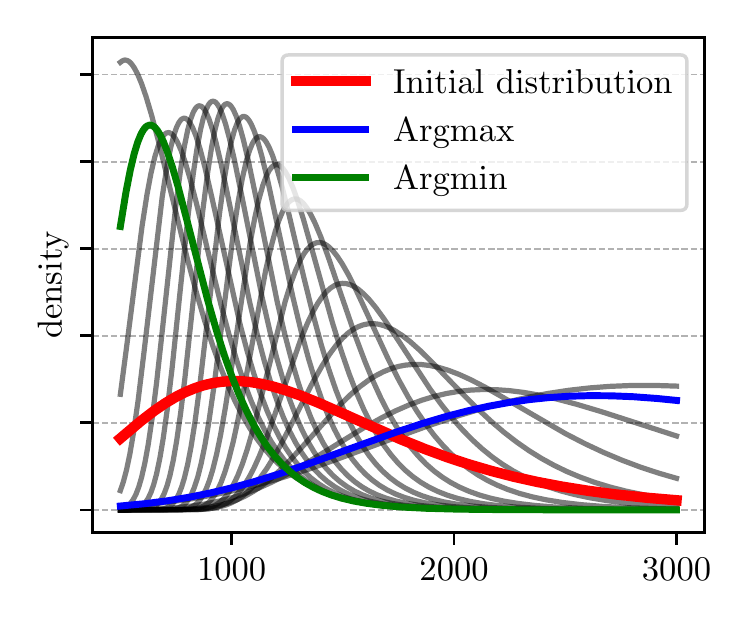}\caption{Densities over the Fisher sphere ($\delta = 1.4$).}
      \label{fig: distrib cas crue 14}
  \end{subfigure}
  \caption{Analysis of the Fisher metric-based perturbation of the truncated Gumbel distribution of the variable $Q$ (see Table \ref{tab: table1}).}
  \label{fig: gumbel}
\end{figure}

The results of the OF-PLI, displayed in Figure $\ref{fig: flood}$, confirm those of the target Sobol indices (see Table~\ref{tab: sobol}): the variables $3$ and $4$, corresponding to $Z_v$ and $Z_m$, are much less influential on the output quantile of level $\alpha=0.95$ than the variables $1$ and $2$, corresponding to $Q$ and $K_s$.
\begin{figure}[!ht]
    \centering
    \includegraphics[scale =1]{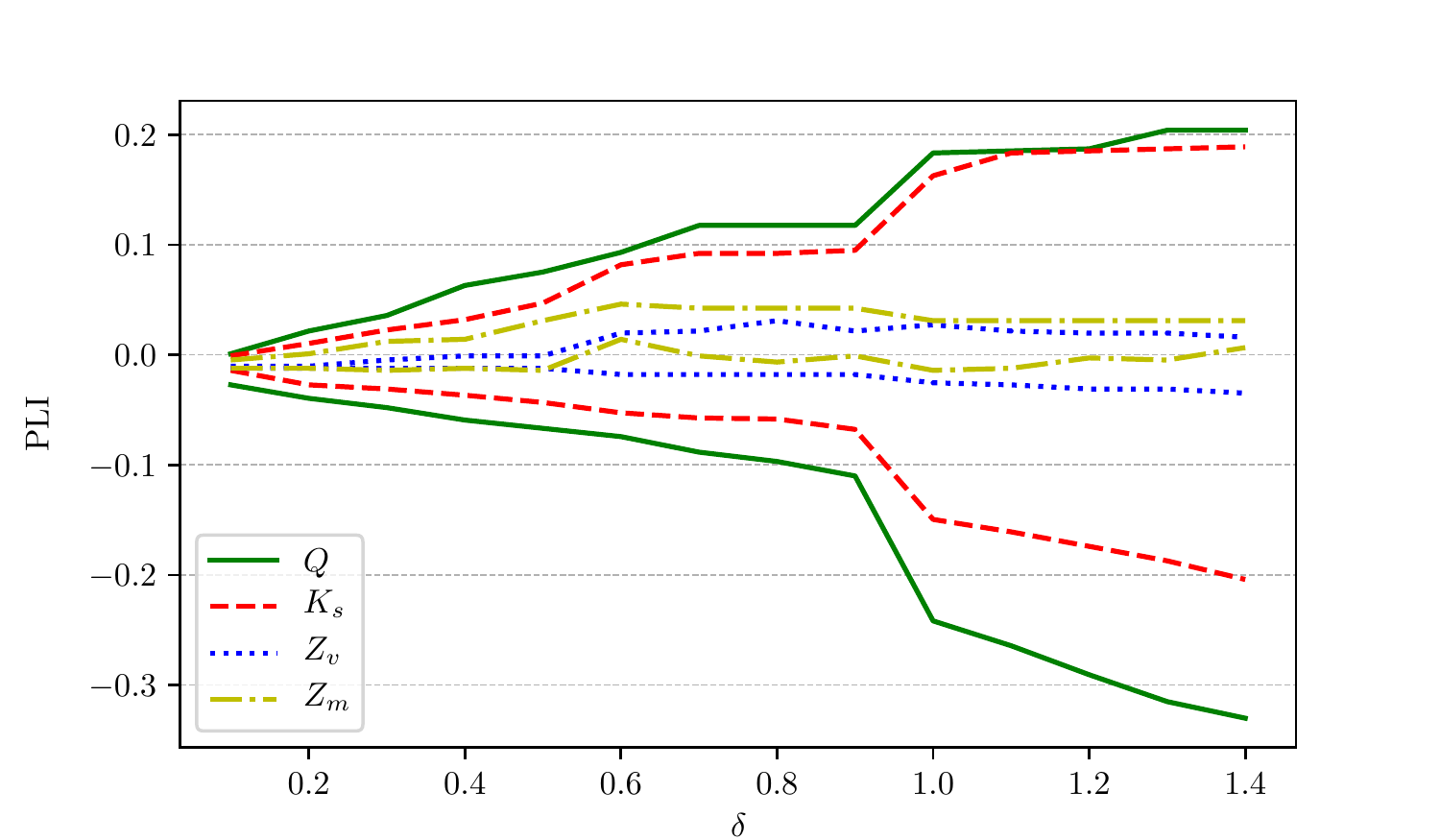}
    \caption{Maximum and minimum estimated value of the OF-PLI $\widehat{S}^+_{N,i\delta}$ and $\widehat{S}^-_{N,i\delta}$ for the different variables of the flood model.}
    \label{fig: flood}
\end{figure}
It turns out that in this case study, perturbations of $Q$ and $K_s$ appear to have comparable effects on the $95\%$-quantile of $H$, although they contribute quite differently to the output variance. 
On the other hand, compared to target Sobol indices, OF-PLI provide more informative results with their evolution as a function of $\delta$.
This clearly shows how a lack of knowledge on an imperfectly known input may turn out to have anything from a high to low impact on the value of a risk measure. 
Note also that the flat segments visible on some curves are due to approximation errors attributed to the low number of sample points $N$ and the high quantile level ($0.95$).
In conclusion, this example confirms the interest of OF-PLI in that it conveys useful complementary information to existing sensitivity indices.

%%%%%%%%%%%%%%%%%%%%%%%%%%%%%%%%%%%%%%%%%

\end{document}